\documentclass[11pt]{amsart}

\usepackage[colorlinks]{hyperref}
\usepackage{enumitem}
\newcommand{\mylabel}[2]{#2\def\@currentlabel{#2}\label{#1}}
\makeatother

\usepackage[left=1in,right=1in,top=1in,bottom=1in]{geometry}

\usepackage{amsmath}
\usepackage{amssymb}
\usepackage{amsthm}
\usepackage{amstext}
\usepackage{amsopn}
\usepackage{placeins}

\usepackage{amsfonts}
\usepackage{apptools}

\usepackage{subcaption}
\usepackage{tikz}
\usetikzlibrary{shapes.geometric}

\usepackage{graphicx}
\usepackage{bm}

\newtheorem{theorem}{Theorem}[section]
\newtheorem{lemma}[theorem]{Lemma}
\newtheorem{proposition}[theorem]{Proposition}

\theoremstyle{definition}
\newtheorem{definition}[theorem]{Definition}
\newtheorem{example}[theorem]{Example}
\newtheorem{remark}[theorem]{Remark}

\numberwithin{equation}{section}

\newcommand{\ba}{\begin{array}}
\newcommand{\ea}{\end{array}}

\begin{document}

\title[Growth Rate and Biomass in Stream Networks]{Maximizing Metapopulation Growth Rate and Biomass in Stream Networks}

\maketitle

\author{Tung D. Nguyen\footnote{Department of Mathematics, Texas A\&M University, College Station, TX 77843, USA}, Yixiang Wu\footnote{Department of Mathematical Sciences, Middle Tennessee State University, Murfreesboro, Tennessee 37132, USA}, Amy Veprauskas\footnote{Department of Mathematics, University of Louisiana at Lafayette, Lafayette, LA 70501, USA}, Tingting Tang\footnote{Department of Mathematics and Statistics, San Diego State University, San Diego, CA 92182, USA}, Ying Zhou\footnote{Department of Mathematics, Lafayette College, Easton, PA 18042, USA}, Charlotte	Beckford\footnote{Department of Mathematics, University of Tennessee, Knoxville, TN 37916, USA}, Brian Chau\footnote{Department of Mathematical and Statistical Sciences, University of Alberta, Edmonton, AB T6G 2G1, Canada}, Xiaoyun Chen\footnote{Department of Mathematics and Statistics, University of North Carolina at Charlotte, Charlotte, NC 28223, USA}, Behzad Djafari Rouhani\footnote{Department of Mathematical Sciences, University of Texas at El Paso, El Paso, TX 79968, USA},  Yuerong Wu\footnote{Department of Mathematics and Statistics, Georgia State University, Atlanta, GA 30303, USA}, Yang	Yang\footnote{Department of Mathematics, The Ohio State University, Columbus, OH 43210, USA},  and Zhisheng Shuai\footnote{Department of Mathematics, University of Central Florida, Orlando, Florida 32816, USA}}

\begin{abstract}
We consider the logistic metapopulation model over a stream network and use the metapopulation growth rate and the total biomass (of the positive equilibrium) as metrics for different aspects of population persistence. Our objective is to find distributions of resources that maximize these persistence measures. We begin our study by considering stream networks consisting of three nodes and  prove that the strategy to maximize the total biomass is to concentrate all the resources in the most upstream locations. In contrast, when the diffusion rates are sufficiently small,  the metapopulation growth rate is maximized when all resources are concentrated in one of the most downstream locations.  These two main results are generalized to stream networks with any number of patches.   
\end{abstract}


\section{Introduction}

Streams, rivers (streams of large water bodies) and creeks (streams of small water bodies) are important components of freshwater ecosystems. One of their key features is the unidirectional current of flowing water (lotic) due to gravity, in contrast with standing bodies of water in lakes and ponds (lentic). This characterizes two kinds of movement of biological organisms in a stream: random diffusion and directed drift. It is of both theoretical and practical significance to understand the joint impact of diffusion and drift on the persistence of an ecosystem, which is the core of management and conservation problems in freshwater habitats \cite{mckenzie2012r_0,pachepsky2005persistence,ramirez2012population,speirs2001persistence}.

Recent research on population persistence in stream networks have highlighted the impact of driving factors such as the size of habitat (patch, stream, watershed) \cite{carrara2014,HILDERBRAND2003,lewis2011patchsize, tamario2021}, connectivity \cite{carrara2014,fagan2002connectivity,samia2015,tamario2021}, human influences \cite{peoples2011} and flow regime \cite{fagan2002connectivity,liu2021asymptotics,mazari2022shape, samia2015,speirs2001persistence}. Many questions and challenges still remain due to the complexity of the ecosystem in stream networks, and rigorous theoretical results are rare, even for small networks.

In this paper we use directed graphs to describe stream networks, where individuals are assumed to live in patches (or nodes, vertices) and the weights of the edges indicate the movement rate of individuals between patches. These patches can be thought of as a section of a stream in the stream system. We  start our work by considering the logistic population model in 
simple stream networks of three nodes (these networks were considered in two recent publications \cite{jiang2020two, Jiang-Lam-Lou2021}).  Later, we  define {\it leveled graphs} and {\it homogeneous flow stream networks} of $n$ nodes, which are generalizations of the simple stream networks of three nodes. Moreover, we  generalize our results on the  metapopulation logistic model to these stream networks. 

We are interested in how to maximize the persistence of a single species living in a stream network system. In each patch, we suppose that  the population follows a logistic type growth functional, where the environmental carrying capacity is assumed to be a constant. The growth rates of individuals in each patch are assumed to be dependent on resource availability, 
where the total amount of resources in the stream network is assumed to be a constant. The persistence of the species in the stream network is quantified in terms of two measures: the metapopulation growth rate and the total biomass (of the stable positive equilibrium). Therefore, the problems we consider are to maximize these persistence measures when varying the distribution of resources.

The total biomass for single species population models has been studied extensively \cite{bai2016optimization, deangelis2016dispersal, gao2022total,heo2021fragmentation, inoue2021unboundedness, liang2012dependence,  lou2006effects,   mazari2020optimal,  mazari2022optimisation, mazari2021fragmentation, Nagahara2021,nagahara2018maximization,  zhang2017carrying,zhang2015effects}.  In a work by Lou \cite{lou2006effects} on the diffusive logistic model, it was observed that the total biomass of the single species may exceed the total carrying capacity, which has motivated a series of related works \cite{bai2016optimization, deangelis2016dispersal,  heo2021fragmentation, inoue2021unboundedness, liang2012dependence,zhang2017carrying, zhang2015effects}.  In particular, the ratio  of total biomass and total resources is bounded in a one dimensional reaction-diffusion model \cite{bai2016optimization} and unbounded when the spatial dimension is greater than one \cite{inoue2021unboundedness}.  The distribution of resources realizing the total biomass has been shown to be of the bang-bang type \cite{ding2010optimal,  mazari2020optimal, mazari2022optimisation,mazari2021fragmentation, nagahara2018maximization}. 
Maximizing the total biomass for a patch model with logistic growth and  random movement has also been considered \cite{liang2021optimal,Nagahara2021}.
However, we note that in most of these studies, the population growth rate in each patch is assumed to be proportional to the carrying capacity, unlike the  model considered in this paper in which we choose to fix the carrying capacity while varying patch growth rates.

Maximizing the population growth rate for single species 
reaction-diffusion logistic models has also been studied in the literature \cite{cantrell1989diffusive, cantrell1991diffusive, lamboley2016properties,lou2006minimization}, and  the solution was also found to be of bang-bang type. 
Recent studies considered more general reaction-diffusion-advection models \cite{liu2021asymptotics, mazari2022shape}.  We are not aware of any parallel studies in patch logistic models. A closely related work \cite{arino2019number} studied the minimal number of patches with positive growth rate needed to achieve a positive growth rate for the metapopulation model. More general studies have considered the relationship between the dispersal rate and the metapopulation growth rate in a patchy environment  \cite{altenberg2012resolvent,chen2022two, karlin1982classifications, kirkland2006evolution}. A key finding of these studies is that, in a heterogeneous environment with a single dispersal mechanism, the metapopulation growth rate is a decreasing function of the propensity to disperse. 

 The following biological insights are highlighted by our  studies on metapopulation models over stream networks:
\begin{enumerate}
 \item[(i)] {\it Concentration of resources in one of the most downstream patches tends to increase the metapopulation growth rate in stream networks.} We use perturbation arguments to study two cases: (1) small diffusion rates; (2) uniformly distributed resources.  For both cases, we provide evidences to suggest that increasing resources in the downstream patches yields the largest increase in growth rate (see Theorems \ref{theorem growth rate}, \ref{theorem growth rate 2}, and \ref{theorem:growth_n}).

\item[(ii)]  {\it Concentration of resources in the most upstream patches tends to increase the population biomass in stream networks.} We use a sign pattern argument to rigorously prove that the total biomass is maximized for stream networks of three patches when the resources are concentrated in the most upstream patches (see Theorem \ref{theorem:BM}).  This result is generalized to arbitrary stream networks using   monotone dynamical system arguments (see Theorem \ref{theorem:biomass_n}). 

\item[(iii)]{\it  A larger drift-diffusion ratio could promote population persistence.} We provide some evidences to show that the  drift-diffusion ratio $q/d$ could promote the metapopulation growth rate and biomass. This is very different from the observations in \cite{ chen2022invasion,chen2021}, which state that  the species with a smaller drift rate $q$ or a larger diffusion rate $d$ wins the competition in a two-species Lotka-Volterra competition model over stream networks. 
\end{enumerate}

The paper is organized as follows.  In Section \ref{section-meta}, we define the logistic stream network model under consideration and revisit recent theoretical results for species in spatially heterogeneous environments. We then focus on the persistence problem of a single species in stream networks in Sections \ref{section-three}
and \ref{section-biomass}. In particular, we examine the question of maximizing the metapopulation growth rate and network biomass, respectively, in stream networks consisting of three nodes. In Section \ref{section-n}, we extend our results to  a general stream network of $n$ nodes. Finally, in Section \ref{section-conclusion} we summarize our results and discuss possible extensions of this work.

\section{Model formulation and preliminary results}
\label{section-meta}

In this section, we revisit the metapopulation logistic model  and propose the questions on maximizing the metapopulation growth rate and total biomass for stream networks. Let $n$ be a positive integer representing the fixed number of patches (or nodes) in a heterogeneous environment and $u_i=u_i(t), 1\le i\le n$, denote the population scale (size or density) of a certain species of study in patch $i$ at time $t\ge 0$.  Assume that at each patch $i$, the population follows logistic growth with  intrinsic growth rate $r_i$ and  environmental carrying capacity $K_i>0$. That is, when in isolation (i.e., without population movement between patches), the population satisfies the following system 
\begin{equation}\label{eq-logistic}
u_i' =\frac{du_i}{dt}= r_iu_i\Big(1-\frac{u_i}{K_i}\Big), \quad i=1,\ldots,n.
\end{equation}

Dispersal links local populations in patches together to form a metapopulation. Let $\ell_{ij}\ge 0$ denote the movement rate of the individuals from patch $j$ to patch $i$, for $1\le i,j\le n$ and $i\not= j$. 
We always assume $\ell_{ii}=0$ for all $i$. The system describing the dynamics of this metapopulation takes the following form:
\begin{equation}\label{eq-system}
u_i' = r_iu_i\Big(1-\frac{u_i}{K_i}\Big) + \sum_{j=1}^n \big(\ell_{ij} u_j - \ell_{ji}u_i\big), \quad i=1,\ldots,n.
\end{equation} 
The first term in the sum above, $\sum_{j} \ell_{ij} u_j$, tracks all incoming movements (flux in) to patch $i$ while the second term, $\sum_j \ell_{ji}u_i$, sums all outgoing movements (flux out) departing from patch $i$.

All movement coefficients in \eqref{eq-system} can be associated with a \textit{movement network} $G$.  Mathematically, $G$ is a weighted, directed graph (digraph) which consists of  $n$ nodes (each node $i$ in $G$ corresponds to patch $i$ in the heterogeneous environment). In $G$, there is a directed edge (arc) from node $j$ to node $i$ if and only if $\ell_{ij}>0$ and in addition, we assign $\ell_{ij}$ as the weight of the arc. Thus we also denote such a movement network as $(G, L)$, where the $n\times n$ \textit{connection matrix} $L$, whose off-diagonal entries are $\ell_{ij}$ and diagonal entries are $-\sum_j \ell_{ji}$, is as follows:

\begin{equation}\label{eq-L}
L:=\begin{pmatrix}
-\sum_j \ell_{j1} & \ell_{12} & \cdots & \ell_{1n}\\
\ell_{21} & -\sum_j \ell_{j2} & \cdots & \ell_{2n}\\
\vdots & \vdots & \ddots & \vdots\\
\ell_{n1} & \ell_{n2} & \cdots & -\sum_j \ell_{jn}
\end{pmatrix}.
\end{equation}
Throughout this paper, we assume that the movement network $G$ is strongly connected, i.e.,  $L$ is irreducible.  It is easy to see that $(1,1, \dots, 1)$ is a left eigenvector of $L$ corresponding to eigenvalue $0$. By the Perron-Frobenius Theorem, 0 is a simple eigenvalue of $L$ corresponding with a positive eigenvector.

The \textit{metapopulation growth rate} $\rho$ of \eqref{eq-system}, determining the metapopulation \linebreak growth rate when the population is small, is defined as the spectral bound of the Jacobian matrix $J$ for the linearization of \eqref{eq-system} at the trivial equilibrium $(u_1,u_2,\ldots,u_n)=(0,0,\ldots,0)$. That is, 
\begin{equation}\label{eq-r}
\rho :=\max\Big\{\mathrm{Re} \lambda : \; \lambda \ \mbox{is an eigenvalue of } J=L+R \Big\}, 
\end{equation} 
where $R=\mathrm{diag}\{r_i\}$ and $L$ is the connection matrix \eqref{eq-L}. An upper bound and a lower bound for the metapopulation growth rate $\rho$ are established in \cite{chen2022two}:

\begin{proposition}[\cite{chen2022two}]\label{proposition spectral bounds}
Let $L$ as defined in \eqref{eq-L} be irreducible. Suppose  $\bm{r}=(r_1, \dots, r_n)\ge  (\neq) \bm 0$.
The metapopulation growth rate $\rho$ as defined in \eqref{eq-r} has the following bounds:
\begin{equation}
\sum_{i=1}^n \theta_i r_i \le \rho \le \max_i \{r_i\}, 
\end{equation}
where $(\theta_1, \theta_2,\ldots,\theta_n)^\top$ is the  positive eigenvector of $L$ corresponding to eigenvalue 0 with $\sum_{i=1}^n\theta_i=1$. 
\label{thm-chen}
\end{proposition}

\begin{remark} In fact, the result from \cite{chen2022two} is more general than Proposition \ref{proposition spectral bounds}. The Jacobian matrix of the model in \cite{chen2022two} has the form$\mu L+R$, where $\mu$ is a positive coefficient. 
Let $s(\mu L+R)$ denote the spectral bound of a matrix $\mu L+R$. By \cite{chen2022two}, the metapopulation growth rate, given by $s(\mu L+R)$, is either a constant or strictly decreasing in $\mu>0$ on $(0, \infty)$  with
$$
\lim_{\mu\to 0}s(\mu L+R)= \max_i \{r_i\} \ \ \text{and} \ \  \lim_{\mu\to \infty}s(\mu L+R)= \sum_{i=1}^n \theta_i r_i.
$$
Proposition \ref{proposition spectral bounds} is a special case of this result, where we take $\mu = 1$. 
\end{remark}

The global dynamics of system \eqref{eq-system} is well-known:
\begin{proposition}[\cite{cosner1996variability,li2010global,Lu1993}]\label{propG}
Let $L$ as defined in \eqref{eq-L} be irreducible. Suppose that $\bm{r}=(r_1, \dots, r_n)\ge  (\neq) \bm 0$ and $K_i>0$ for all $1\le i\le n$.  Then  system \eqref{eq-system} has a unique positive equilibrium which is globally asymptotically stable with respect to all nonnegative nontrivial initial data. 
\end{proposition}

Let $\mathbf{u}^*=(u_1^*, u_2^*, \ldots, u_n^*)$ be the unique positive equilibrium of \eqref{eq-system}.
In order to quantitatively measure the metapopulation, we define the \textit{network biomass} $\mathcal{K}$ as 
\begin{equation}\label{eq-K}
\mathcal{K} := \sum_{i=1}^n u_i^* .
\end{equation}
Notice that when there is only one patch in the network (i.e., $n=1$),  the metapopulation growth rate $\rho$ becomes the intrinsic growth rate and the network biomass $\mathcal{K}$ becomes the environmental carrying capacity. If the species has the same intrinsic growth rate at each patch (i.e., $r_1=r_2=\cdots=r_n=r/n$), then $\rho=r/n$, regardless of the movement network.

Suppose that the intrinsic growth rate in each patch $r_i$ depends in a simple, linear way on some resource and the total amount of resources among all patches are fixed, i.e., $r=\sum_{i=1}^n r_i$ is a positive constant.   Here we interpret a positive $r_i$ to represent additional resources in patch $i$, as $r_i\ge 0$ implies that all patches have inherent resources available since $r_i= 0$ means the patch density would remain constant in the absence of dispersal. What type of allocation of these resources results in a maximized growth rate or biomass of the metapopulation? 

We study this question for stream networks in this paper. As a starting point, we  consider three different configurations of stream networks with three nodes, as depicted in Figure~\ref{fig-3patch}. Configuration (i) is commonly observed in the upper course of freshwater systems, where small streams join up to form a larger one. This configuration can also be used to describe the situation when a tributary (i.e., a freshwater stream) feeds into a larger stream (or river). Thus we call configuration (i) a {\it tributary stream}. Configuration (ii) represents a {\it straight stream}, which is often seen in the middle course of stream/river systems (meanders are common but ignored in our study). Configuration (iii) describes a {\it distributary stream}, which is commonly seen in the lower course of the water system (e.g., near the delta of streams/rivers).

We restrict our study to the special case when all patches have the same carrying capacity but could have different intrinsic growth rates. We assume that the  movement of individuals among patches are subject to diffusion and drift, where $d>0$ represents the diffusion magnitude and $q>0$ represents the drift magnitude (See Figure \ref{fig-3patch}). More precisely, we impose the following assumptions:
\begin{enumerate}[label=\textnormal{(H\arabic*)}]

\item $r_i\ge 0$ for all $i$ and $\sum_i r_i =r>0$ is a fixed constant.\label{hp1}
\item For all $i$, $K_i=K>0$.\label{hp2}
\item The connection matrix $L$ is irreducible. Moreover, all upstream movement coefficients are given by $d$ and downstream movement coefficients are $d+q$.\label{hp3}
\end{enumerate}

We also define the {\it drift-diffusion ratio} of a stream as $q/d$. Later we show that an increase in this drift-diffusion ratio could promote the network biomass if source patches are located upstream.

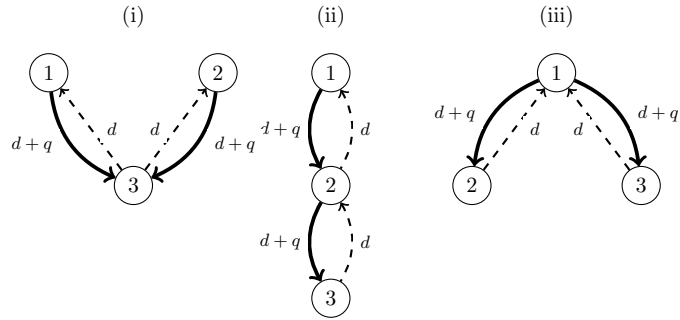
\begin{figure}[htbp]
\centering
\begin{tikzpicture}[scale=0.75, transform shape]
\begin{scope}[every node/.style={draw}, node distance= 1.5 cm]

    \node[draw=white] (i) at (-3.5,1) {(i)};
    \node[draw=white] (ii) at (0,1) {(ii)};
    \node[draw=white] (iii) at (4,1) {(iii)};

    \node[circle] (1) at (0,0) {$1$};
    \node[circle] (2) at (0,-2) {$2$};
    \node[circle] (3) at (0,-4) {$3$};

    \node[circle] (4) at (4,0) {$1$};
    \node[circle] (5) at (2.5,-2) {$2$};
    \node[circle] (6) at (5.5,-2) {$3$};

    \node[circle] (7) at (-5,0) {$1$};
    \node[circle] (8) at (-2,0) {$2$};
    \node[circle] (9) at (-3.5,-2) {$3$};
\end{scope}
\begin{scope}[every node/.style={fill=white},
              every edge/.style={thick}]

    \draw[line width=0.5mm] [->](1) to [bend right] node[left=0.1] {{\footnotesize $d+q$}} (2);
    \draw[line width=0.5mm] [->](2) to [bend right] node[left=0.1] {{\footnotesize $d+q$}} (3);
    \draw[thick, dashed] [<-](1) to [bend left] node[right=0.1] {{\footnotesize $d$}} (2);
    \draw[thick, dashed] [<-](2) to [bend left] node[right=0.1] {{\footnotesize $d$}} (3);

    \draw[line width=0.5mm] [->](4) to [bend right] node[left=5] {{\footnotesize $d+q$}} (5);
    \draw[thick, dashed] [->](5) to node[right=4] {{\footnotesize $d$}} (4);
    \draw[line width=0.5mm] [->](4) to [bend left] node[right=5] {{\footnotesize $d+q$}} (6);
    \draw[thick, dashed] [->](6) to  node[left=4] {{\footnotesize $d$}} (4);

    \draw[line width=0.5mm] [->](7) to [bend right] node[left=5] {{\footnotesize $d+q$}} (9);
    \draw[thick, dashed] [->](9) to node[right=4] {{\footnotesize $d$}} (7);  
    \draw[line width=0.5mm] [->](8) to [bend left] node[right=5] {{\footnotesize $d+q$}} (9);
    \draw[thick, dashed] [->](9) to node[left=4] {{\footnotesize $d$}} (8);
\end{scope}
\end{tikzpicture}
\caption{Stream networks with three nodes: (i) a tributary stream with 1,2 being upstream nodes and 3 being a downstream node; (ii) a straight stream with 1 being an upstream node, 2 being a middle node and 3 being a downstream node; (iii) a distributary stream with 1 being an upstream node and 2,3 being downstream nodes. Here the solid edges represent a larger movement from the upstream to the downstream and the dashed edges represent movement from the downstream to the upstream. The individuals are assumed to be subject to a diffusion rate $d$ and a drift rate $q$.
}\label{fig-3patch}
\vskip -20pt
\end{figure}

\section{Maximizing the metapopulation growth rate for stream networks of three patches}\label{section-growthrate}
\label{section-three}
In this section, we study the configuration of $\bm r=(r_1, r_2, r_3)$ to maximize the metapopulation growth rate $\rho$. We first compute the bounds of  $\rho$ for the three patch networks in Figure~\ref{fig-3patch}. These three networks describe all possible homogeneous flow stream networks, as defined later in Section \ref{section-n}, on three nodes.

(i) The tributary stream of three patches in Figure~\ref{fig-3patch} has  connection matrix
$$
L=\begin{bmatrix}
  -d-q & 0 & d \\
  0 & -d-q & d \\
  d+q & d+q & -2d \\
 \end{bmatrix}.
 $$
It follows from Proposition~\ref{thm-chen} that the metapopulation growth rate $\rho$ for \eqref{eq-system} on the tributary stream is bounded as follows:
$$\frac{dr_1 + dr_2 + (d+q)r_3}{3d+q} \le \rho\le \max_i \{r_i\},$$
since $(\frac{d}{3d+q}, \frac{d}{3d+q}, \frac{d+q}{3d+q})$ is the normalized eigenvector corresponding to the eigenvalue 0 of $L$. Since $r_1+r_2+r_3=r$ and $r_i\ge 0$ (i.e., assumption (H1)), the above inequalities can be rewritten as
$$
\frac{1}{3+\frac{q}{d}}r+\frac{\frac{q}{d}}{3+\frac{q}{d}}r_3 \le \rho\le \max_i \{r_i\}\le r.
$$

(ii) The straight stream of three patches in Figure~\ref{fig-3patch}  has connection matrix
$$ 
L=\begin{bmatrix}
  -(d+q) & d & 0 \\
  d+q & -2d-q & d \\
  0 & d+q & -d \\
\end{bmatrix}.
$$
By Proposition~\ref{thm-chen}, the metapopulation growth rate  has the following bounds:
$$
\frac{1}{3+3\frac{q}{d}+\frac{q^2}{d^2}}r+\frac{\frac{q}{d}}{3+3\frac{q}{d}+\frac{q^2}{d^2}}r_2+\frac{2\frac{q}{d}+\frac{q^2}{d^2}}{3+3\frac{q}{d}+\frac{q^2}{d^2}}r_3\le \rho \le \max_i \{r_i\}\le r.
$$

(iii) The distributary stream  in Figure~\ref{fig-3patch} has connection matrix
$$ 
L=\begin{bmatrix}
  -2d-2q & d & d \\
  d+q & -d & 0 \\
  d+q & 0 & -d \\ 
\end{bmatrix}.
$$
By Proposition~\ref{thm-chen}, the metapopulation growth rate $\rho$ satisfies 
$$
\frac{1}{3+2\frac{q}{d}}r + \frac{\frac{q}{d}}{3+2\frac{q}{d}}r_2 + \frac{\frac{q}{d}}{3+2\frac{q}{d}}r_3\le \rho \le \max_i \{r_i\}\le r.
$$

\begin{remark}
The connection matrix $L$ for each of the three cases can be written as $L=dD+qQ$, where $D$ and $Q$ represent the diffusion and advective movement patterns, respectively.  
\end{remark}

Notice that the lower bounds for $\rho$ are maximized if all the resources are concentrated at the downstream ends. By Remark 2.2, this lower bound is achieved when $d$ and $q$ approach infinity while maintaining a constant ratio. We conjecture that $\rho$ is maximized for the tributary and straight stream if $\bm r=(0, 0, r)$ and for the distributary stream if $\bm r=(0, r, 0)$ or $\bm r=(0, 0, r)$. Next, we provide two pieces of evidence to support this conjecture.   We first verify that the conjecture holds when diffusion is sufficiently small. We then consider the case where the system is perturbed away from uniformly distributed resources.  In what follows, we use the notation  $\rho(d, q, \bm r)$ to emphasize the dependence of $\rho$ on $d, q$ and $\bm r$.

\subsection{{ Small diffusion, $d\gtrapprox 0$}}  First suppose $d=0$. Then we can easily see:
\begin{itemize}

          \item [(i)] For a tributary stream, $\rho(0, q, \bm r)=\max\{r_1-q, r_2-q, r_3\}$. Therefore for any $q>0$,  the maximum of $\rho(0, q, \bm r)$ is achieved when $\bm r=(0, 0, r)$.
          
    \item[(ii)] For a straight stream, $\rho(0, q, \bm r)=\max\{r_1-q, r_2-q, r_3\}$. Therefore for any $q>0$,  the maximum of $\rho(0, q, \bm r)$ is achieved when $\bm r=(0, 0, r)$.
    
       \item [(iii)] For a distributive stream, $\rho(0, q, \bm r)=\max\{r_1-2q, r_2, r_3\}$. Therefore for any $q>0$,  the maximum of $\rho(0, q, \bm r)$ is achieved when $\bm r=(0, r, 0)$ or $\bm r=(0, 0, r)$.

\end{itemize}
Thus, when $d=0$, the metapopulation growth rate is maximized when all the resources are concentrated in a single downstream patch.
We can extend this result to hold for $d\gtrapprox 0$.

\begin{theorem}\label{theorem growth rate}
Suppose that assumptions (H1) and (H3) hold for system \eqref{eq-system}. Let $q, r>0$ be fixed. Then the following statements hold:
\begin{enumerate}
\item[\rm (i)] There exists $d^*>0$ such that the maximum metapopulation growth rate $\rho$ of the tributary or straight stream  in Figure \ref{fig-3patch} is attained when $(r_1, r_2, r_3)=(0, 0, r)$ for all $0<d\le d^*$;
\item[\rm (ii)] There exists $d^{**}>0$ such that the maximum metapopulation growth rate $\rho$ of the distributary  stream  in Figure \ref{fig-3patch} is attained when $(r_1, r_2, r_3)=(0, r, 0)$ or $(r_1, r_2, r_3)=(0, 0, r)$ for all $0<d\le d^{**}$.
\end{enumerate}
\end{theorem}

\begin{proof}
We only prove (i) here. The proof of (ii) is similar and is provided in the appendix.  Let $\rho(d, q, \bm r )$  be the metapopulation growth rate. Denote $S=\{(r_1, r_2, r_3)\in\mathbb{R}^3_+: r_1+r_2+r_3=r\}$, $S_1=\{(r_1, r_2, r_3)\in S: r_1\ge q/2 \ \ \text{or}\ \ r_2\ge q/2\}$, $S_2=\overline{S\backslash S_1}$, and $\hat S_2=\{(r_1, r_2, r_3)\in S: r_1<q \ \ \text{and}\ \ r_2< q\}$.

If $\bm r\in S_1$, then $r_1-q, r_2-q\le r-q$ and $r_3=r-r_1-r_2\le r-q/2$. Hence, we have 
$$
\rho(0, q, \bm r)=\max\{r_1-q, r_2-q, r_3\}\le r-q/2.
$$
By the continuity of $\rho$ \cite{zedek1965continuity}, there exists $d_1>0$ such that 
\begin{equation}\label{ess1}
\rho(d, q, \bm r)<r-q/4
\end{equation}
for all $d\in [0, d_1]$ and $\bm r\in S_1$. Since $\rho(0, q, (0, 0, r))=r$, there exists $d_2<d_1$ such that 
\begin{equation}\label{ess2}
\rho(d, q, (0, 0, r))>r-q/4
\end{equation} 
for all $d\in [0, d_2]$. By \eqref{ess1}-\eqref{ess2} and $(0, 0, r)\in S_2$, we have 
$$
\max\{\rho(d, q, \bm r): \bm r\in S\}= \max\{\rho(d,q, \bm r ): \bm r\in S_2\}
$$
for all $d\in [0, d_2]$, i.e., the maximum of $\rho$ is attained in $S_2$. 

For any $\bm r\in \hat S_2$, we have $\rho(0, q, \bm r)=r_3$, where $r_3$ is a simple eigenvalue of $dD+qQ+\text{diag}\{r_i\}$. Therefore, there exists $d_3<d_2$ such that $\rho(d, q, \bm r)$ is analytic for $d\in [0, d_3]$ and $\bm r\in S_2$. Hence, the derivatives of $\rho$ are continuous for $d\in [0, d_3]$ and $\bm r\in  S_2$.

Let $h=(-1, 0, 1)$ or $(0, -1, 1)$. It is easy to compute the directional derivatives of $\rho$ with respect to $(r_1, r_2, r_3)$
$$
D_h\rho(0, q, \bm r)=1, \ \text{for any}\ \bm r\in S_2.
$$
By continuity, there exists $d^*<d_3$ such that $D_h\rho(d, q, \bm r)>0$ for any $0\le d\le d^*$ and $\bm r\in S_2$. Hence, the maximum of $\rho(d, q, \bm r)$ is attained at $\bm r=(0, 0, r)$ for any $0\le d\le d^*$. 
\end{proof}

\subsection{Equal patch growth rates, $r_1 = r_2 = r_3= r/3$}
	The second case we consider is when all patches have the same intrinsic growth rate, that is $r_1 = r_2 = r_3= r/3$. Notice that, since the column sums of both $D$ and $Q$ are zero, we have $(1, 1, 1)(dD + qQ) = 0$. It follows that 
	\[(1, 1, 1)\left(d D + q Q + r/3 I \right) = (1, 1, 1) r/3. \]
	Thus, $r/3$ is an eigenvalue of the inherent projection matrix and $(1,1,1)$ is the corresponding left eigenvector. Next we note that since $dD + qQ$ is essentially nonnegative and irreducible, the matrix $d D + q Q + r/3 I$ is as well. Thus, by the Perron-Frobenius theorem, we have $\rho(d, q, \bm r)=r/3$. 	
 
 In Theorem \ref{theorem growth rate 2} we consider how the metapopulation growth rate is affected when the patch growth rates are perturbed away from uniform by examining the impact of increasing the amount of resources in a single patch. Since we wish to keep the total amount of resources fixed at $r$, we assume that an increase in one patch is compensated by a decrease in the other patches. Theorem \ref{theorem growth rate 2} shows that the metapopulation growth rate increases the most if the increase in resources is applied to one of the most downstream nodes.

\begin{theorem}\label{theorem growth rate 2}
Suppose that assumptions (H1) and (H3) hold for system \eqref{eq-system}. Let $q, r>0$ be fixed and let $r_i = r/3$. Consider the perturbation $rI \rightarrow rI + \epsilon E$ where $\epsilon \gtrapprox 0$ and $E$ is a diagonal matrix with diagonal entries $1$, $-\alpha $ and $-\beta $ such that $\alpha, \beta \geq 0$ and $\alpha + \beta = 1$. Then the metapopulation growth rate is largest when matrix entry 1 corresponds to one of the most downstream patches.
\end{theorem}

\begin{proof}
  Let  $\rho_\epsilon$ denote the metapopulation growth rate  of the perturbed matrix $rI + \epsilon E$.  Following standard perturbation arguments (see for example \cite{kirkland2013group}, Chapter 3), we have 
	\begin{equation}\label{specrtal radius equal rs perturbation}
		\rho_\epsilon \approx  r/3 +\epsilon{\bm w}^\intercal E {\bm v}  + o(\epsilon^2),
	\end{equation}
	where ${\bm v}$ and ${\bm w}^\intercal$ are the right and left Perron eigenvectors of $d D + q Q + r/3 I$, respectively, such that ${\bm w}^\intercal {\bm v} = 1.$ Here we take ${\bm w}^\intercal = (1, 1, 1)$. Notice that the right eigenvector ${\bm v}$ is a solution to the equation $(d D + q Q){\bm v} = {\bf 0}.$ Denote this right eigenvector for the tributary, straight, and distributary streams by ${\bm v}_t,{\bm v}_s,$ and ${\bm v}_d$, respectively. Explicit calculations show
{\small
\begin{align*}
{\bm v}_t = \frac{d}{3d +q}\left(\begin{array}{c} 1\\ 1\\ \frac{d+q}{d} \end{array} \right), \;
{\bm v}_s =  \frac{d^2}{3d^2+3dq+q^2}\left(\begin{array}{c} 1\\ \frac{d+q}{d}\\ \left(\frac{d+q}{d}\right)^2 \end{array} \right), \;
{\bm v}_d = \frac{d}{3d+2q }\left(\begin{array}{c} 1\\ \frac{d+q}{d}\\ \frac{d+q}{d} \end{array} \right) .
\end{align*}
}
Notice that, for all three cases, the largest component of ${\bm v}$ corresponds to the most downstream patch(es).

We first consider the straight stream configuration. Let $E_i$ denote the diagonal perturbation matrix describing an increase in resources in patch $i$, that is, $E_{1} = \text{diag}\{1, -\alpha, -\beta\}$, $E_{2} = \text{diag}\{ -\alpha,1, -\beta\}$, and $E_{3} = \text{diag}\{-\alpha, -\beta, 1\}$. We  find:
\begin{align*}
&  {\bm w}^\intercal E_1 {\bm v}_s = \frac{d^2}{3d^2+3dq+q^2} \left( 1 - \alpha \frac{d+q}{d} - \beta \frac{(d+q)^2}{d^2}  \right)<0, \\
&  {\bm w}^\intercal E_2 {\bm v}_s = \frac{d^2}{3d^2+3dq+q^2} \left( -\alpha   + \frac{d+q}{d} - \beta \frac{(d+q)^2}{d^2}  \right),\\
&  {\bm w}^\intercal E_3 {\bm v}_s = \frac{d^2}{3d^2+3dq+q^2} \left( -\alpha  - \beta \frac{d+q}{d}   + \frac{(d+q)^2}{d^2}  \right) >0.
	\end{align*}
Therefore, increasing the intrinsic growth rate in patch 1 while maintaining the total intrinsic growth rate at $r$ results in a decrease in the metapopulation growth rate. Meanwhile, increasing the intrinsic growth rate in patch 3 increases the metapopulation growth rate. Moreover, the metapopulation growth rate when the intrinsic growth rate is increased in patch 2 lies between these two cases. Thus, we may conclude that shifting  resources from upstream patches to the most downstream patch increases the the metapopulation growth rate {the most}.
In the same manner, we may show that this also holds true for the other two configurations.
\end{proof}

\section{Maximizing the network biomass for stream networks of three patches}\label{section-biomass}

In the following theorem, we summarize our results on maximizing the network biomass associated with stream networks in Figure \ref{fig-3patch}. The proof is provided in the next subsection. 

\begin{theorem}\label{theorem:BM}
Suppose that assumptions (H1), (H2) and (H3) hold for system \eqref{eq-system} on stream networks. Let $\mathcal{K}_t,\mathcal{K}_s,\mathcal{K}_d$ denote the network biomass associated with a tributary stream, a straight stream, and a distributary stream as depicted in Figure \ref{fig-3patch}, respectively. Then the following results hold.
\begin{enumerate}
    \item[\rm(i)] $\mathcal{K}_t \leq (3+\frac{q}{d})K$. The maximum is reached when $(r_1,r_2,r_3)= (r_1,r-r_1,0)$ for any $0\leq r_1\leq r$.
    \item[\rm(ii)] $\mathcal{K}_s \leq (3+3\frac{q}{d}+\frac{q^2}{d^2})K$. The maximum is reached when $(r_1,r_2,r_3)= (r,0,0)$.
    \item[\rm(iii)] $\mathcal{K}_d \leq (3+2\frac{q}{d})K$. The maximum is reached when $(r_1,r_2,r_3)=(r,0,0)$.
\end{enumerate}
\end{theorem}

Theorem~\ref{theorem:BM} implies that the straight and distributary streams, which are commonly observed in the middle and lower courses of freshwater systems, could sustain larger population sizes than the tributary streams, which often occur in the upper course of stream/river systems.
Moreover, the biased population movement in the streams tends to promote the network biomass. Specifically, when the drifted movement is neglected (i.e., $q=0$), all upper bounds of the network biomass $\mathcal{K}$ in Theorem~\ref{theorem:BM} become $3K$, the sum of the carrying capacities of the three patches. In contrast, the larger the drift-diffusion ratio $q/d$, the larger the maximum network biomass.

\subsection{Proof of Theorem \ref{theorem:BM} Using Sign Patterns}

In order to prove Theorem \ref{theorem:BM}, 
we introduce the notion of sign patterns of an equilibrium. Note that the positive equilibrium $\bm u^* = (u_1^*,u_2^*,u_3^*)$ satisfies
\begin{equation}\label{equilibrium}
 r_iu^*_i\bigg(1-\frac{u^*_i}{K}\bigg) + \sum_{j=1}^3(\ell_{ij}u_j^* -\ell_{ji}u_i^*)  = 0, \quad i=1,2,3,
\end{equation}
where $\ell_{ij}$ are the entries of the movement matrix $L$. 

\begin{definition}\label{sign_pattern}
Suppose $\bm u^*$ is the unique positive equilibrium for a certain choice of parameters $(r_1,r_2,r_3)$. We assign each node $i$ of the stream network with one of the three signs $\{+,-,0\}$, which is the same as the sign of the corresponding logistic growth term  at the equilibrium, i.e.,  $$sign(\mathrm{node} \ i) = sign\Bigg(r_iu^*_i\Big(1-\frac{u^*_i}{K}\Big)\Bigg).$$ The sign pattern of the equilibrium $\bm u^*$ is defined as
\[
sign(\bm u^*) = (sign(\mathrm{node} \ 1),sign(\mathrm{node} \ 2),sign(\mathrm{node} \ 3)).
\]
\end{definition}

\begin{definition}
Suppose that node $i$ is an immediate upstream node of node $j$. At the positive equilibrium we define the {flow} from node $i$ to node $j$ as $(d+q)u_i^*$ and the {flow} from node $j$ to node $i$ as $du_j^*$. We say that there is a net flow from node $i$ to node $j$ if $(d+q)u_i^*>du_j^*$, there is a net flow from node $j$ to node $i$ if $(d+q)u_i^* < du_j^*$, and there is zero net flow between the 2 nodes if $(d+q)u_i^*=du_j^*$. 
\end{definition}
\begin{example}
For the straight stream as depicted in Figure \ref{fig-3patch}, if $(r_1,r_2,r_3)=(r,0,0)$, then the positive equilibrium is 
\[
\bm u^* = \bigg(K,\frac{d+q}{d}K,\frac{(d+q)^2}{d^2}K\bigg).
\]
Therefore the sign pattern of this equilibrium is $sign(\bm u^*)=(0,0,0)$. We show later that $(0,0,0)$ is the sign pattern which maximizes the network biomass in all three stream networks (straight, tributary, and distributary).
\end{example}
We provide some intuitions on sign patterns in the following remark.
\begin{remark}\label{remark:sign}
A node having a $(-)$ sign implies that the population of that node at the equilibrium exceeds the carrying capacity $K$. This also implies that at the equilibrium there is a {\it net inflow} of the population from adjacent nodes to this node, or the total flow out of the node is smaller than the total flow into the node. 

On the other hand, a patch having a $(+)$ sign implies that the population of that node at the equilibrium is smaller than the carrying capacity $K$ and there is a {\it net outflow} of the population from this node to adjacent nodes.

Finally, a patch having a $(0)$ sign implies that the population of that node at the equilibrium is exactly the same as the carrying capacity $K$ or $r_i=0$ for that node, and thus both the growth and the net flow at the patch are 0.
\end{remark}

Next, we make an observation that due to the strong connectivity and asymmetry of stream networks, not all sign patterns can happen. Specifically, we say that a sign pattern is \textit{admissible} if there is a choice of parameters $(r_1,r_2,r_3)$ such that the corresponding positive equilibrium admits that sign pattern; otherwise, it is called \textit{inadmissible}. The following lemma characterizes the admissible/inadmissible patterns for stream networks.

\begin{lemma}\label{lemma:sign}
The following statements hold for all three stream networks as depicted in Figure \ref{fig-3patch}.
\begin{enumerate}
\item[\rm (i)] Any admissible sign pattern must be $(0,0,0)$ or it must have at least one $(+)$ and one $(-)$ sign.
\item[\rm (ii)] The population at the positive equilibrium of the most upstream nodes cannot exceed the carrying capacity $K$, and thus the sign patterns where
a most upstream node has a $(-)$ sign are inadmissible. Specifically, for a tributary stream, the sign patterns $(-,\cdot,\cdot)$ and $(\cdot,-,\cdot)$ are inadmissible. For a straight stream, the sign patterns $(-,\cdot,\cdot)$ are inadmissible. For a distributary stream, the sign patterns $(-,\cdot,\cdot)$ are inadmissible.
\item[\rm (iii)] The sign patterns where a most downstream node has a $(+)$ sign are inadmissible. Specifically, for a tributary stream, the sign patterns $(\cdot,\cdot,+)$ are inadmissible. For a straight stream, the sign patterns $(\cdot,\cdot,+)$ are inadmissible. For a distributary stream, the sign patterns $(\cdot,\cdot,+)$ and $(\cdot,+,\cdot)$ are inadmissible. 
\end{enumerate}
\end{lemma}

\begin{proof}
\phantom{}

(i) Taking the sum of the three equations in \eqref{equilibrium}, we have
\begin{equation}\label{eq:samesigns}
\sum_{i=1}^3 r_i u_i^* \bigg(1-\frac{u_i^*}{K}\bigg) =0.
\end{equation}
Thus either all 3 terms in the sum are 0, or there must be at least one positive and one negative term.

(ii) We prove the statement for each of the three configurations. 

\vspace{.1in}
\noindent \textbf{A tributary stream}
\vspace{.1in}

First, without loss of generality, assume that $u_1^* >K$. This means node 1 must have a $(-)$ or $(0)$ sign.  Remark \ref{remark:sign} tells us that the total flow into node 1, $du_3^*$, must be greater than or equal to the total flow from node 1, $(d+q)u_1^*$. This implies $du_3^* \geq (d+q)u_1^*$ and thus $u_3^* \geq \frac{d+q}{d}u_1^* >\frac{d+q}{d}K$.  Thus node 3 must have a $(-)$ or $(0)$ sign. From part (i), node 2 must have a $(+)$ sign or the sign pattern must be $(0,0,0)$. 

We first consider the case when the sign pattern is $(0,0,0)$. Since there must be at least one $r_i>0$, there must be at least one $u_i^*=K$. It is easy to check that when $u_2^*=K$ or $u_3^*=K$, we must have $u_1^* < K$. Thus in all cases, $u_1^*\leq K$, which contradicts the initial assumption.

If node 2 has a $(+)$ sign, there must be a net outflow from node 2. This means $u_2^*<K$ and $(d+q)u_2^*>du_3^*$, which imply $u_3^*< \frac{d+q}{d}u_2^*< \frac{d+q}{d}K$. We reach a contradiction.

\vspace{.1in}
\noindent \textbf{A straight stream}
\vspace{.1in}

Next, suppose that $u_1^* > K$, which implies node 1 must have a $(-)$ or $(0)$ sign. Since the total flow into node 1, $du_2^*$, must be greater than or equal to the total flow from node 1, $(d+q)u_1^*$, this implies $du_2^* \geq (d+q)u_1^*$ and thus $u_2^* \geq \frac{d+q}{d}u_1^* >K$.  Thus node 2 must have a $(-)$ or $(0)$ sign. 
Using a similar argument, we must have $(d+q)u_1^*+du_3^* \geq du_2^*+(d+q)u_2^*$. Combining this with $du_2^*\geq (d+q)u_1^*$ from the analysis of node 1 yields $du_3^* \geq (d+q)u_2^*$, which implies $u_3^* \geq \frac{d+q}{d}u_2^* > K$. Thus node 3 also has a $(-)$ or $(0)$ sign. 

According to part (i), since there is no node with a $(+)$ sign, the only admissible sign pattern in this case is $(0,0,0)$. Since there must be at least one $r_i>0$, there must be at least one $u_i^*=K$. It is easy to check that when $u_2^*=K$ or $u_3^*=K$, we must have $u_1^* < K$. Thus in all cases, $u_1^*\leq K$, which contradicts the initial assumption.

\vspace{.1in}
\noindent \textbf{A distributary stream}
\vspace{.1in}

Finally, assume that $u_1^* >K$, or equivalently node 1 must have a $(-)$ or $(0)$ sign. From \eqref{eq:samesigns} in part (i), the sign pattern must be $(0,0,0)$, or at least one of the 2 downstream nodes must have a $(+)$ sign. 

If the sign pattern is $(0,0,0)$, similar to the first two configurations, we must have $u_1^* \leq K$ which contradicts the initial assumption.

If the sign pattern is not $(0,0,0)$, at least one of the 2 downstream nodes must have a $(+)$ sign.  Without loss of generality, suppose node 2 has a $(+)$ sign. Then we have $u_1^*>K$ and $u_2^*<K$ and there must be a net outflow from node 2. This requires $du_2^*>(d+q)u_1^*$, and thus $u_2^*>\frac{d+q}{d}u_1^*>K$, which contradicts the assumption that node 2 has a $(+)$ sign.

(iii) We prove the statement for each of the three configurations. 

For a tributary stream, suppose that node 3 has a $(+)$ sign. From part (i), either node 1 or node 2 must have a $(-)$ sign, which is in turn impossible due to part (ii).

For a straight stream, suppose that node 3 has a $(+)$ sign. We must have $u_3^*<K$, and there is a net outflow from node 3. This means $du_3^*>(d+q)u_2^*$, and thus $u_2^*<\frac{d}{d+q}u_3^*<\frac{d}{d+q}k<K$. So node 2 must have a $(+)$ or a $(0)$ sign and due to part (i), node 1 must have a $(-)$ sign, which is in turn impossible due to part (ii).

For a distributary stream, without loss of generality, suppose that node 2 has a $(+)$ sign. Using the same argument above, $u_1^*<\frac{d}{d+q}K$ and node 1 must have a $(+)$ or  $(0)$ sign. Due to equation \eqref{eq:samesigns}, node 3 must have a $(-)$ sign. This means there must be a net inflow into node 3, which implies $du_3^*<(d+q)u_1^*$ and thus $u_3^*<\frac{d+q}{d}u_1^*<K$. We reach a contradiction.

\end{proof}

\begin{remark}\label{remark:upstream_downstream}
The last two statements in Lemma \ref{lemma:sign} highlight the effect of the asymmetry inherent to stream network on the positive equilibrium. Since the asymmetry is in favor of the downstream nodes, at the equilibrium it is impossible for a most upstream node to have population exceeding the carrying capacity or a most downstream node to have population less than the carrying capacity.
\end{remark}

Now we are ready to list all the admissible sign patterns for each stream network, and provide the proof of Theorem \ref{theorem:BM}.
\begin{proof}[Proof of Theorem \ref{theorem:BM}]
It suffices to prove the theorem for the sign patterns not ruled out by Lemma \ref{lemma:sign}.

\vspace{.1in}
\noindent{(i) A tributary stream:} The remaining admissible sign patterns for a tributary stream are included in Figure \ref{sign-tributary} together with a visualization of the net population flow into or out of a node at the positive equilibrium.

\begin{figure}[htbp]
\centering
\begin{tikzpicture}[scale=0.75, transform shape]
\begin{scope}[every node/.style={draw}, node distance= 1.5 cm]
    \node[circle] (11) at (0,0) {$0$};
    \node[circle] (21) at (0.7,-2) {$0$};
    \node[circle] (31) at (1.4,0) {$0$};
    
    \node[circle] (12) at (3,0) {$+$};
    \node[circle] (22) at (3.7,-2) {$-$};
    \node[circle] (32) at (4.4,0) {$0$};
    
    \node[circle] (13) at (6,0) {$0$};
    \node[circle] (23) at (6.7,-2) {$-$};
    \node[circle] (33) at (7.4,0) {$+$};
    
    \node[circle] (14) at (9,0) {$+$};
    \node[circle] (24) at (9.7,-2) {$-$};
    \node[circle] (34) at (10.4,0) {$+$};
\end{scope}
\begin{scope}[every node/.style={fill=white},
              every edge/.style={thick}]

        \draw[thick] [-] (11) -- (21);
        \draw[thick] [-] (21) -- (31);
        
        \draw[thick] [->] (12) -- (22);
        \draw[thick] [-] (32) -- (22);
        
        \draw[thick] [-] (13) -- (23);
        \draw[thick] [->] (33) -- (23);
        
        \draw[thick] [->] (14) -- (24);
        \draw[thick] [->] (34) -- (24);

\end{scope}
\end{tikzpicture}
\caption{All admissible sign patterns for a tributary stream. The arrows capture the net flows between nodes at the positive equilibrium.}\label{sign-tributary}
\vskip -15pt
\end{figure}
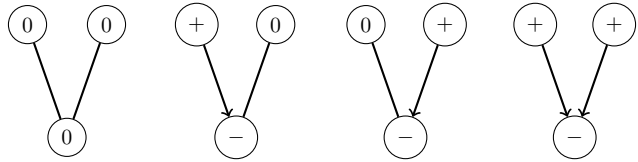
It is straightforward to check that the only choices of $(r_1,r_2,r_3)$ resulting in the sign pattern $(0,0,0)$ are $(r_1,r_2,r_3)=(r_1,r_2,0)$ and $(0,0,r)$. The former yields $\bm u^* = \big(K,K,\frac{d+q}{d}K\big)$ and the latter yields $\bm u^* = \big(\frac{d}{d+q}K,\frac{d}{d+q}K,K\big)$. Thus $\mathcal{K}_t\leq \big(3+\frac{d+q}{d}\big)K$ and the equality happens when $(r_1,r_2,r_3)=(r_1,r_2,0)=(r_1,r-r_1,0)$ for any $0\leq r_1\leq r$.

For the remaining sign patterns, again by using net flow at each node it is easy to show that $u_1^* <K$ , $u_2^*<K $ and $u_3^*< \frac{d+q}{d}K$. Thus for these sign patterns we have $\mathcal{K}_t\leq \big(3+\frac{d+q}{d}\big)K.$

\vspace{.1in}
\noindent{(ii) A straight stream:} The remaining admissible sign patterns for a straight stream are included in Figure \ref{sign-straight} together with a visualization of the net population flow into or out of a node at the positive equilibrium.

\begin{figure}[htbp]
\centering
\begin{tikzpicture}[scale=0.75, transform shape]
\begin{scope}[every node/.style={draw}, node distance= 1.5 cm]
    \node[circle] (11) at (0,0) {$0$};
    \node[circle] (21) at (0,-2) {$0$};
    \node[circle] (31) at (0,-4) {$0$};
    
    \node[circle] (12) at (2,0) {$0$};
    \node[circle] (22) at (2,-2) {$+$};
    \node[circle] (32) at (2,-4) {$-$};
    
    \node[circle] (13) at (4,0) {$+$};
    \node[circle] (23) at (4,-2) {$0$};
    \node[circle] (33) at (4,-4) {$-$};
    
    \node[circle] (14) at (6,0) {$+$};
    \node[circle] (24) at (6,-2) {$-$};
    \node[circle] (34) at (6,-4) {$0$};
    
    \node[circle] (15) at (8,0) {$+$};
    \node[circle] (25) at (8,-2) {$+$};
    \node[circle] (35) at (8,-4) {$-$};
    
    \node[circle] (16) at (10,0) {$+$};
    \node[circle] (26) at (10,-2) {$-$};
    \node[circle] (36) at (10,-4) {$-$};
\end{scope}
\begin{scope}[every node/.style={fill=white},
              every edge/.style={thick}]
    \draw[thick] [-] (11) --  (21);
    \draw[thick] [-] (21) --  (31);

    \draw[thick] [-] (12) --  (22);
    \draw[thick] [->] (22) --  (32);
    
    \draw[thick] [->] (13) --  (23);
    \draw[thick] [->] (23) --  (33);
    
    \draw[thick] [->] (14) --  (24);
    \draw[thick] [-] (24) --  (34);
    
    \draw[thick] [->] (15) --  (25);
    \draw[thick] [->] (25) --  (35);
    
    \draw[thick] [->] (16) --  (26);
    \draw[thick] [->] (26) --  (36);
\end{scope}
\end{tikzpicture}
\caption{All admissible sign patterns for a straight stream. The arrows capture the net flows between nodes at the positive equilibrium.}\label{sign-straight}
\vskip -15pt
\end{figure}
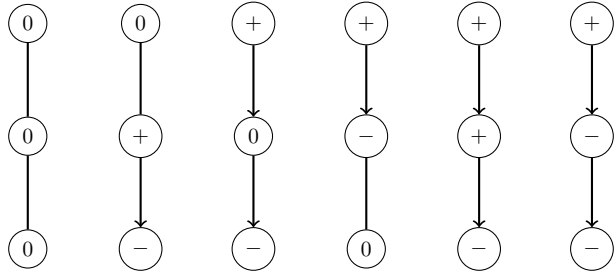
It is straightforward to check that the only choices of $(r_1,r_2,r_3)$ that result in the sign pattern $(0,0,0)$ are $(r_1,r_2,r_3)= (r,0,0), (0,r,0)$ or $(0,0,r)$. In all 3 cases, the positive equilibrium can be computed explicitly, and it is easy to check that $\mathcal{K}_s\leq \big(1+\frac{d+q}{d}+\frac{(d+q)^2}{d^2}\big)K = \big(3+3\frac{q}{d}+\frac{q^2}{d^2}\big)K$
where the equality happens when $(r_1,r_2,r_3) = (r,0,0)$.
For the sign pattern $(0,+,-)$, we have $u_2^*<K$, and looking at the net flow yields $u_1^*=\frac{d}{d+q}u_2^* <K$ and $u_3^* < \frac{d+q}{d} u_2^* < \frac{d+q}{d} K$. Thus $\mathcal{K}_s <  \big(3+3\frac{q}{d}+\frac{q^2}{d^2}\big)K$.

For each of the remaining sign patterns, by analyzing the net flow at each patch, it is straightforward to verify that $u_1^*<K$, $u_2^*<\frac{d+q}{d}K$ and $u_3^*<\frac{(d+q)^2}{d^2}K$. Thus for these sign patterns we also have $\mathcal{K}_s <  \big(3+3\frac{q}{d}+\frac{q^2}{d^2}\big)K$. 

\vspace{.1in}
\noindent{(iii) A distributary stream:} The remaining admissible sign patterns for a distributary stream are included in Figure \ref{sign-distributary} together with a visualization the net population flow into or out of a node at the positive equilibrium.

\begin{figure}[htbp]
\centering
\begin{tikzpicture}[scale=0.75, transform shape]
\begin{scope}[every node/.style={draw}, node distance= 1.5 cm]

    \node[circle] (11) at (0.7,0) {$0$};
    \node[circle] (21) at (0,-2) {$0$};
    \node[circle] (31) at (1.4,-2) {$0$};
    
    \node[circle] (12) at (3.7,0) {$+$};
    \node[circle] (22) at (3,-2) {$0$};
    \node[circle] (32) at (4.4,-2) {$-$};
    
    \node[circle] (13) at (6.7,0) {$+$};
    \node[circle] (23) at (6,-2) {$-$};
    \node[circle] (33) at (7.4,-2) {$0$};
    
    \node[circle] (14) at (9.7,0) {$+$};
    \node[circle] (24) at (9,-2) {$-$};
    \node[circle] (34) at (10.4,-2) {$-$};
\end{scope}
\begin{scope}[every node/.style={fill=white},
              every edge/.style={thick}]

        \draw[thick] [-] (11) -- (21);
        \draw[thick] [-] (11) -- (31);
        
        \draw[thick] [-] (12) -- (22);
        \draw[thick] [->] (12) -- (32);
        
        \draw[thick] [->] (13) -- (23);
        \draw[thick] [-] (13) -- (33);
        
        \draw[thick] [->] (14) -- (24);
        \draw[thick] [->] (14) -- (34);

\end{scope}
\end{tikzpicture}
\caption{All admissible sign patterns for a distributary stream. The arrows capture the net flows between nodes at the positive equilibrium.}\label{sign-distributary}
\vskip -20pt\end{figure}
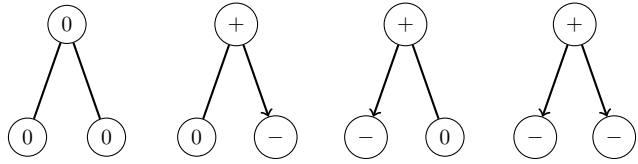

It is straightforward to check that the only choices of $(r_1,r_2,r_3)$ resulting in the sign pattern $(0,0,0)$ are $(r_1,r_2,r_3)= (r,0,0), (0,r_2,r_3)$. The former has positive equilibrium $\bm u^* = \big(K,\frac{d+q}{d}K,\frac{d+q}{d}K\big)$ and thus $\mathcal{K}_d=\big(3+2\frac{q}{d}\big)K$. The latter has positive equilibrium $\bm u^* = \big(\frac{d}{d+q}K,K,K\big)$ and thus $\mathcal{K}_d<\big(3+2\frac{q}{d}\big)K$.

For each of the remaining sign patterns, using the net flow at each node, it is easy to verify that $u_1^* <K$, $u_2^*<\frac{d+q}{d}K$ and $u_3^* < \frac{d+q}{d}K$. Thus for these sign patterns we must have $\mathcal{K}_d<\big(3+2\frac{q}{d}\big)K$.
\end{proof}

\begin{remark}
From Figures \ref{sign-tributary}, \ref{sign-straight}, and \ref{sign-distributary} we observe that the net flows at the positive equilibrium in all admissible sign patterns tend to agree with the flows of the stream network from upstream to downstream. In addition, at the positive equilibrium the populations in upstream nodes are generally smaller than the populations in downstream nodes. 
\end{remark}

We illustrate the contrast between maximizing the metapopulation growth rate and maximizing the network biomass in Figure \ref{fig:biomass_total} which compares the network biomass for the three-node stream networks when the resources are concentrated either in a most upstream or downstream patch. While a strategy that maximizes the growth rates results in faster initial growth, it yields a substantially smaller network biomass.
\begin{figure}[htbp]
    \centering
     \begin{subfigure}[b]{0.32\textwidth}
         \centering
          \includegraphics[width=1.3in,height=1.3in]{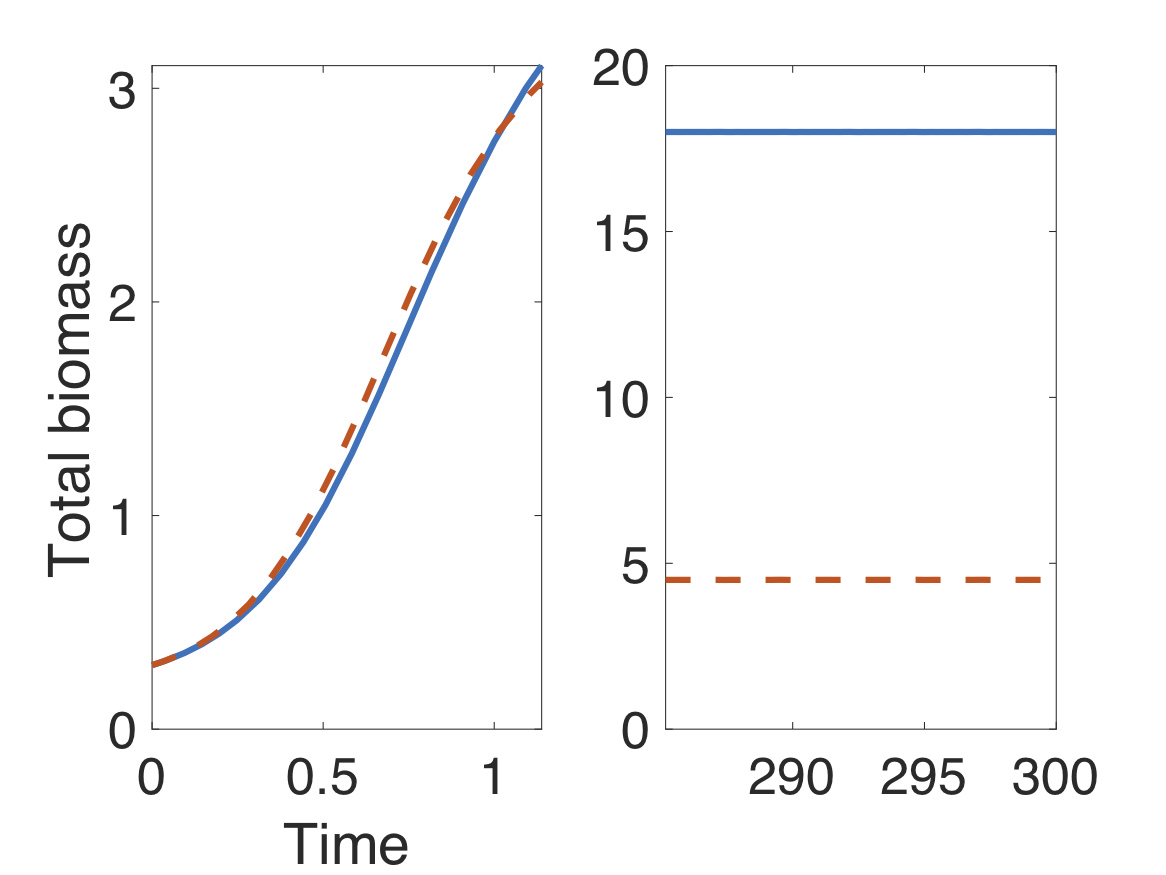}
         \caption{}
         \label{fig:a}
     \end{subfigure}
     \begin{subfigure}[b]{0.32\textwidth}
        \centering
         \includegraphics[width=1.3in,height=1.3in]{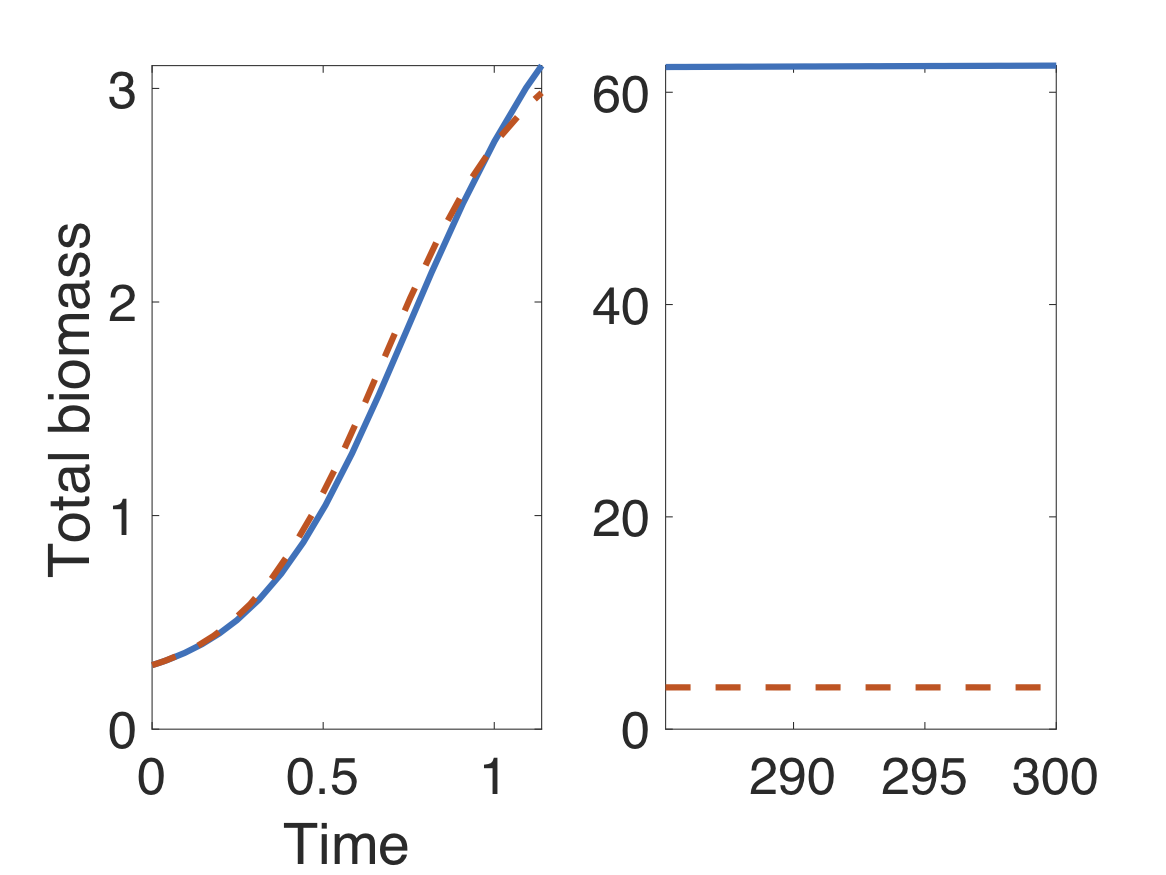}
         \caption{}
         \label{fig:b}
     \end{subfigure}
      \begin{subfigure}[b]{0.32\textwidth}
         \centering
          \includegraphics[width=1.3in,height=1.3in]{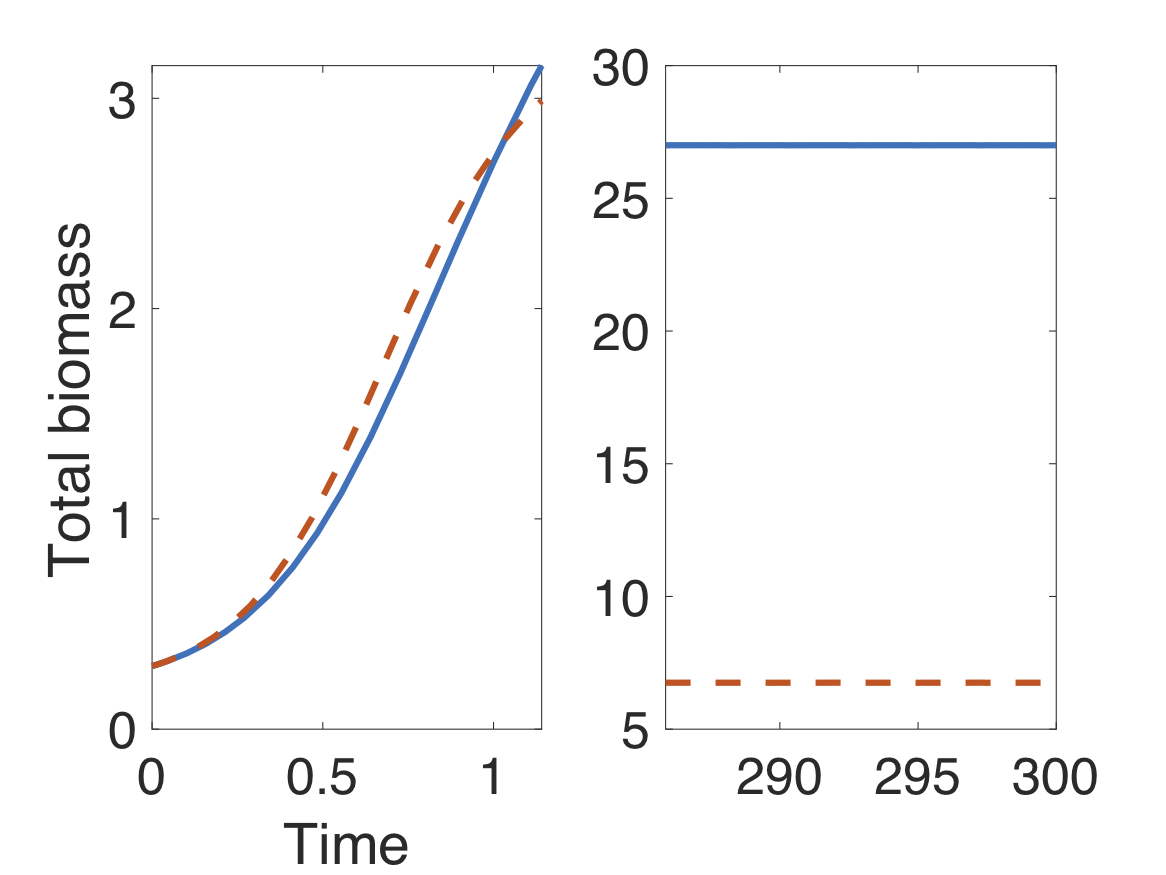}
         \caption{}
         \label{fig:c}
     \end{subfigure}
    \caption{Total population (biomass) over time for (a) the tributary stream, (b) the straight stream, and (c) the distributary stream with two distributions of resources: $(r_1,r_2,r_3)=(r,0,0)$ which maximizes the network biomass (solid blue line) and $(r_1,r_2,r_3)=(0,0,r)$ which maximizes the growth rate (dashed red line). In all the simulations, we set $d=0.1, q=0.3,K=3$, and $r=5$.}
    \label{fig:biomass_total}
    \vskip -20pt
\end{figure}

\section{Maximizing the metapopulation growth rate and network biomass for stream networks of $n$ patches}
\label{section-n}
In this section we provide a definition of stream networks to allow for an arbitrary number of patches. We then extend the results on the metapopulation growth rate and network biomass in Section \ref{section-growthrate} and Section \ref{section-biomass} to stream networks with $n$ patches. Additionally, we provide some further analysis on the straight stream network with $n$ patches, which highlights further differences between the distribution maximizing the metapopulation growth rate and the distribution maximizing the total biomass.

\subsection{Leveled graphs and homogeneous flow stream networks}
In order to define stream networks with $n$ patches, we first introduce the notion of \textit{leveled graphs}.

\begin{definition}\label{def:leveled_graph}
Let $G$ be a directed graph, and denote the set of nodes of $G$ by $V$. Consider a function $f: V\to \mathbb{Z}_{\geq 0}$. We call $f$ a \textit{level function} and for each node $i$, we define $f(i)$ as the \textit{level} of node $i$. We say that $(G,f)$ is a \textit{leveled graph} if the following assumptions are satisfied
\begin{enumerate}%
    \item[(i)] For each $0\leq k\leq \max_{i\in V}\{f(i)\}$, there exists a node $j$ such that $f(j)=k$.
    \item[(ii)] For each pair of nodes $i, j$, there is no edge between $i$ and $j$ if $|f(i)-f(j)|\neq 1$.
\end{enumerate}
\end{definition}

The definition of a leveled graph requires two conditions on the stream network. The first condition is that every node is assigned a level with each level containing at least one node. The second condition is that two nodes may only be connected if they are in consecutive levels. Notice that any graph without cycles of length greater than two can be made into a leveled graph. A graph with cycles of length greater than two can be made into a leveled graph if any path in a cycle connecting the node with the highest level and the node with the lowest level has the same length. For example, the graph depicted in Figure~\ref{leveled_graph}(a) is a leveled graph while the one in Figure~\ref{leveled_graph}(b) cannot be a leveled graph for any choice of level functions.

In what follows, we use leveled graphs to model stream networks. Specifically, we define the nodes with level $0$ to be most upstream nodes, and the level of an arbitrary node to be  the distance between that node and the most upstream nodes. In particular, the nodes with maximum level are the most downstream nodes.

\begin{definition}
Consider a leveled graph $(G,f)$ and an irreducible connection matrix $L$. We say that $(G,f, L)$ is a \textit{homogeneous flow stream network} if the following assumptions are satisfied
\begin{enumerate}
    \item[(i)] If there is an edge from node $i$ to node $j$, then there is also an edge from node $j$ to node $i$. 
    \item[(ii)] If there is an edge from node $i$ to node $j$, then the weight is  $\ell_{ji} = d+q$ if $f(j)-f(i)=1$ (i.e., the edge is from an upstream to a downstream node) and $\ell_{ji}=d$ if $f(i)-f(j)=1$ (i.e., the edge is from a downstream to an upstream node). 
\end{enumerate}
\end{definition}
It is easy to check that the three stream configurations in Figure \ref{fig-3patch} are homogeneous flow stream networks. Additionally, we list all possible homogeneous flow stream networks with four nodes (up to vertical symmetry) in Figure \ref{fig2} in the Appendix. Finding the total number of homogeneous flow stream networks with $n$ nodes is an interesting combinatorics question, and we leave it to curious readers.

\begin{definition}
Let $(G,f, L)$ be a homogeneous flow stream network and $V$ be the set of nodes of $G$. 
A node $i\in V$ is called a most downstream end node if $f(i)=\max_{i\in V}\{f(i)\}$; it is called a downstream end node if there does not exist an adjacent node $j\in V$ such that $f(j)-f(i)=1$.
\end{definition}

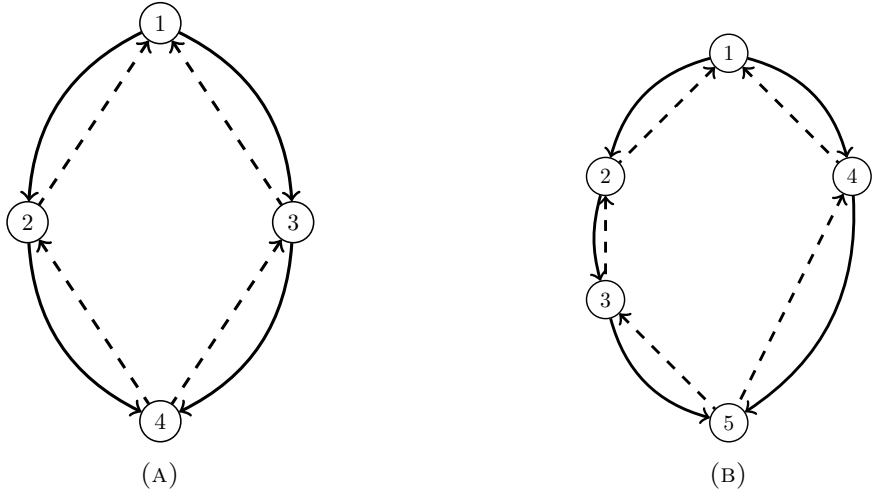
\begin{figure}[htbp]
  \begin{subfigure}[b]{0.45\textwidth}
    \centering
    \resizebox{0.6\linewidth}{!}{

{\small \begin{tikzpicture}[scale=0.6, transform shape]
\begin{scope}[every node/.style={draw}, node distance= 1.5 cm]

    \node[circle] (1) at (-4+0,     0-0) {$1$};
    \node[circle] (2) at (-4-2,   0-3) {$2$};
    \node[circle] (3) at (-4+2,   0-3) {$3$};
    \node[circle] (4) at (-4+0,     0-6) {$4$};  
  
\end{scope}
\begin{scope}[every node/.style={fill=white},
              every edge/.style={thick}]
    
    \draw[thick] [->](1) to [bend right] node[left=5] {{}} (2); 
    \draw[thick, dashed] [<-](1) to node[right=4] {{}} (2); 
    \draw[thick]  [->](1) to [bend left] node[right=5] {{}} (3);
    \draw[thick, dashed] [<-](1) to node[left=4] {{}} (3); 
    \draw[thick] [->](2) to [bend right] node[right=5] {{}} (4);
    \draw[thick, dashed] [<-](2) to node[left=4] {{}} (4); 
    \draw[thick] [->](3) to [bend left] node[right=5] {{}} (4); 
    \draw[thick, dashed] [<-](3) to node[left=4] {{}} (4);

\end{scope}

\end{tikzpicture}
}}
\caption{}
\label{a}
\end{subfigure}
 \begin{subfigure}[b]{0.45\textwidth}
    \centering
    \resizebox{0.6\linewidth}{!}{

\small{
 \begin{tikzpicture}[scale=0.6, transform shape]
\begin{scope}[every node/.style={draw}, node distance= 1.5 cm]
    
    \node[circle] (5) at (4+0,      0-0) {$1$};
    \node[circle] (6) at (4-2,    0-2) {$2$};
    \node[circle] (7) at (4-2,    0-4) {$3$};
    \node[circle] (8) at (4+2,    0-2) {$4$};
    \node[circle] (9) at (4+0,      0-6) {$5$};
 
\end{scope}
\begin{scope}[every node/.style={fill=white},
              every edge/.style={thick}]
    \draw[thick] [->](5) to [bend right] node[left=5] {{}} (6);
    \draw[thick, dashed] [<-](5) to node[right=4] {{}} (6); 
    \draw[thick] [->](6) to [bend right=15] node[right=5] {{}} (7); 
    \draw[thick, dashed] [<-](6) to (7); 
    \draw[thick] [->](5) to [bend left] node[right=5] {{}} (8); 
    \draw[thick, dashed] [<-](5) to node[left=4] {{}} (8); 
    \draw[thick] [->](8) to [bend left] node[right=5] {{}} (9);
    \draw[thick, dashed] [<-](8) to node[left=4] {{}} (9);
    \draw[thick] [->](7) to [bend right] node[right=5] {{}} (9); 
    \draw[thick, dashed] [<-](7) to node[left=4] {{}} (9);

\end{scope}
\end{tikzpicture}
}
}
\caption{ }
\label{b}
\end{subfigure}
\caption{Graph (a) is a leveled graph with a level function $f(1)=0,f(2)=f(3)=1,f(4)=2$, while (b) is not a leveled graph.}\label{leveled_graph}
\vskip -20pt
\end{figure}

\subsection{Maximizing the growth rate for homogeneous flow stream networks}

We first prove the following lemma concerning the right Perron eigenvector ${\bm v}$ to the matrix $dD + qQ$. For the three-node stream networks considered in Sections \ref{section-growthrate} and \ref{section-biomass} this vector appeared in both the biomass and  growth rate calculations. 

\begin{lemma}\label{lemma vector v}
Let $(G,f,L)$ be a homogeneous flow stream network. Let $\bm v$ be the solution to $(dD+qQ)\bm v = \bm 0$. Then $\bm v$ is defined, up to a constant multiple, as
\[
v_i = \left(\frac{d+q}{d}\right)^{f(i)}.
\]
\end{lemma}
\begin{proof}
 By the Perron-Frobenius Theorem, a non-trivial solution of $(dD+qQ)\bm v = \bm 0$ exists, which is an eigenvector of $dD+qQ$ corresponding to eigenvalue $0$.  For each $i$, the $i$-th row of $(dD+qQ)\bm v$ can be rewritten as
\[
\sum_{j:f(i)-f(j)=1}((d+q)v_j-dv_i) \; \;\;+ \sum_{j:f(j)-f(i)=1}(dv_j - (d+q)v_i) 
\]
where the first sum is over all the nodes connected to and upstream of node $i$, and the second sum is over all the nodes connected to and downstream of node~$i$.

Consider a vector $\bm v=(v_1,\ldots,v_n)$ with
\[
v_i = \left(\frac{d+q}{d}\right)^{f(i)}.
\]
Then we have
\begin{align*}
\sum_{j:f(i)-f(j)=1}((d+q)v_j-dv_i) &=\sum_{j:f(i)-f(j)=1}dv_i\left(\frac{d+q}{d}\frac{v_j}{v_i}-1\right) \\
& = \sum_{j:f(i)-f(j)=1}dv_i\left(\frac{d+q}{d}\left(\frac{d+q}{d}\right)^{f(j)-f(i)}-1\right)\\
&= \sum_{j:f(i)-f(j)=1} 0 =0.
\end{align*}
Similarly, we can show $\sum_{j:f(j)-f(i)=1}(dv_j - (d+q)v_i)=0$. Thus each row of $(dD+qQ)\bm v$ is $0$ and we can conclude $(dD+qQ)\bm v=\bm 0$.
\end{proof}

Using Lemma \ref{lemma vector v}, Theorems \ref{theorem growth rate} and \ref{theorem growth rate 2} in Section \ref{section-growthrate} can be extended to any homogeneous flow stream network, as stated in the Theorem \ref{theorem:growth_n}. Note that the result of part (i) does not specify which  downstream end node one has to concentrate all resources on to maximize the metapopulation growth rate. The proof of part (i) is provided in the Appendix. The proof of  part (ii) follows directly from Lemma \ref{lemma vector v} and the same argument as in Theorem \ref{theorem growth rate 2}, and thus is omitted for the sake of brevity.

\begin{theorem}\label{theorem:growth_n}
Let $(G,f,L)$ be a homogeneous flow stream network, and suppose that assumption  \ref{hp1}  holds.
\begin{enumerate}
    \item [\rm (i)] There exists a $\tilde{d}^*>0$ such that for $0<d<\tilde{d}^*$ the metapopulation growth rate of the stream network is maximized when all resources are distributed to exactly one of the  downstream end nodes.		
    \item [\rm (ii)] Suppose $r_i = r/n$. If resources are perturbed so that the amount of resources in a single node are increased while the resources in other nodes are decreased to maintain $\sum r_i = r$, then the metapopulation growth rateTime increases the most if this increase is applied to one of the most downstream end nodes.		
\end{enumerate}
\end{theorem}

\subsection{Maximizing the biomass for homogeneous flow stream networks}
Next we show that, similar to the stream networks with three nodes, to maximize the network biomass of a homogeneous flow stream network we must concentrate the resources in the most upstream nodes.

\begin{theorem}\label{theorem:biomass_n}
Let $(G,f,L)$ be a homogeneous flow stream network and suppose that assumptions  \ref{hp1} and  \ref{hp2} hold. The network biomass $\mathcal{K}$ has the upper bound
\[
\mathcal{K} \leq K\sum_{i=1}^n \left(1+\frac{q}{d}\right)^{f(i)},
\]
and the maximum is achieved when $r_i=0$ for every node $i$ with a positive level, i.e., $f(i)>0$.
\end{theorem}
\begin{proof}
For $\bm \phi=(\phi_1, \dots, \phi_n), \bm \psi=(\psi_1, \dots, \psi_n)\in\mathbb{R}^n$, we write $\bm \phi> \bm \psi$ if  $\bm \phi\ge \bm \psi$ and $\bm \phi\neq \bm \psi$; we denote  $\bm \phi\gg \bm \psi$ if $\phi_i>\psi_i$ for all $1\le i\le n$. Since $L$ is essentially nonnegative and irreducible, the solutions of \eqref{eq-system} generate a strongly monotone dynamical system \cite[Theorem 4.1.1]{smith2008monotone}, that is, if $\bm u_1(t)$ and $\bm u_2(t)$ are both solutions of \eqref{eq-system}, then $\bm u_1(0)>\bm u_2(0)$ implies $\bm u_1(t)\gg \bm u_2(t)$ for all $t>0$. Let $\bm v$ be the normalized positive eigenvector of $dD+qQ$ corresponding to principal eigenvalue 0 as defined in Lemma \ref{lemma vector v}. Let $\bar{\bm u}=K\bm v$. Then $\bar u_i=K$ if $f(i)=0$, and $\bar u_i>K$ if $f(i)>0$.  It is easy to see that 
$$
0\ge r_i\bar u_i\left(1-\frac{\bar u_i}{K}\right)= \sum_{j=1}^n (dD_{ij}+qQ_{ij})\bar u_j+r_i\bar u_i\left(1-\frac{\bar u_i}{K}\right), \ \ i=1,\dots, n.
$$
Therefore, $\bm {\bar u}$ is an upper solution of \eqref{eq-system}, and the solution $\bm u(t)$ of \eqref{eq-system} with initial condition $\bm u(0)=\bm {\bar u}$ is non-increasing and converges to an equilibrium of \eqref{eq-system} \cite[Proposition 3.2.1]{smith2008monotone}, which is the positive equilibrium by Proposition \ref{propG}.
  Moreover, the above inequality has at least one strict sign (which means that $\bm u(t)$ is strictly decreasing) if and only if there exists a node $i$ with $f(i)>0$ such that $r_i>0$. Let $\bm u^*$ be the positive equilibrium of \eqref{eq-system}. Hence,  we have that $\bm u^*\le \bm{\bar u}$ and $\bm u^*=\bm{\bar u}$ if and only if $r_i=0$ for any node $i$ with $f(i)>0$. This proves the result. 
\end{proof}

\begin{example}\label{example:straight_n}
\begin{figure}[h]
\centering
\begin{tikzpicture}
\begin{scope}[node distance= 1.5 cm]

    \node[circle, draw] (1) at (0,6) {$1$};
    \node[circle, draw] (2) at (2,6) {$2$};
    \node (3) at (4,6) {$\cdots$};
    \node[circle, draw] (4) at (6,6) {$n$};

\end{scope}
\begin{scope}[every node/.style={fill=white},
              every edge/.style={thick}]

     \draw[line width=0.5mm] [->](1) to [bend right] node[below=0.1] {{\footnotesize $d+q$}} (2);
     \draw[line width=0.5mm] [->](2) to [bend right] node[below=0.1] {{\footnotesize $d+q$}} (3);
    \draw[thick, dashed] [<-](1) to [bend left] node[above=0.1] {{\footnotesize $d$}} (2);
    \draw[thick, dashed] [<-](2) to [bend left] node[above=0.1] {{\footnotesize $d$}} (3);
     \draw[line width=0.5mm] [->](3) to [bend right] node[below=0.1] {{\footnotesize $d+q$}} (4);
    \draw[thick, dashed] [<-](3) to [bend left] node[above=0.1] {{\footnotesize $d$}} (4);

\end{scope}
\end{tikzpicture}
\caption{A $n$-patch homogeneous flow stream network, which is the generalization of the straight stream in Figure \ref{fig-3patch}. The level function is $f(i)=i-1.$}\label{fig:straight_n}
\vskip -15pt
\end{figure}
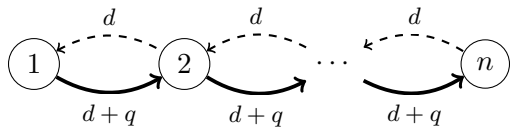

For the homogeneous flow stream network in Figure \ref{fig:straight_n}, we have the upper bound for the network biomass
\[
\mathcal{K} \leq K\sum_{i=0}^{n-1}\left(1+\frac{q}{d}\right)^i
\]
where the maximum is reached when $(r_1,r_2,\dots,r_n)= (r,0,\dots,0)$.
\end{example}

\subsection{Effect of network properties}
From Theorem \ref{theorem:biomass_n}, we know, for example, that a straight stream consisting of $n$ patches can reach a maximum network biomass if $(r_1, r_2, \ldots, r_n)$ $= (r,0,\ldots,0)$. Assuming such a resource distribution, it is clear from the network biomass formula in Example \ref{example:straight_n} that the network biomass increases with respect to the number of patches $n$, and decreases with respect to the diffusion rate $d$. Does the metapopulation growth rate $\rho$ respond to changes in $n$ and $d$ in the same way? In this subsection, we provide additional discussion on how the metapopulation growth rate and network biomass seem to follow opposite trends with respect to changes in parameters such as $n$ and $d$. We base these discussions on an $n$ patch straight stream for illustration purpose.

In Figure \ref{fig:chainlengthB}, the metapopulation growth rate for an $n$-patch straight stream with $(r_1, r_2, \ldots,$ $r_n) = (r,0,\ldots,0)$ is numerically calculated and plotted for different parameter values. In all three sub figures, the solid curves are those of a $2$-patch stream, the dotted curves a $3$-patch stream, the dashed curves a $4$-patch stream, and the dash-dotted curves a $5$-patch stream. These numerical results imply that as $n$ increases from 2 to 5, the metapopulation growth rate decreases. This is expected, since larger $n$ means that there is time for individuals in the population to spend outside of the patch with a positive growth rate. 
\par

Besides the contrast between the metapopulation growth rate and the network biomass when $n$ increases, the dependence of the metapopulation growth rate on $d$ is also more complex than that of the total biomass. The three sub figures in Figure \ref{fig:chainlengthB} correspond to different $q$ values. When $q$ is small ($q = 0.5$, Fig. \ref{fig:chainlengthB}(a)), the metapopulation growth rate decrease with respect to $d$ in all curves. When $q$ is larger ($q = 1.5$, Fig \ref{fig:chainlengthB}(b)), the metapopulation growth rate may increase, or decrease before increasing, depending on the value of $n$. When $q$ is large ($q = 10$, Fig \ref{fig:chainlengthB}(c)), the metapopulation growth rate increases with respect to $d$. Therefore, when a directional drift is present, the metapopulation growth rate is no longer always decreasing with respect to the diffusion rate $d$ as discussed in Remark 2.2.

\begin{figure}
     \centering
     \begin{subfigure}[b]{0.32\textwidth}
         \centering
          \includegraphics[scale=.17]{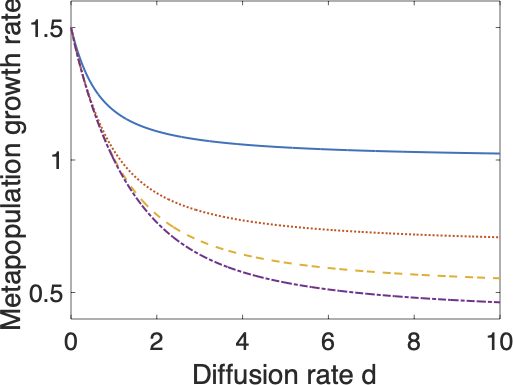}
         \caption{}
         \label{fig:a1}
     \end{subfigure}
     \begin{subfigure}[b]{0.32\textwidth}
        \centering
         \includegraphics[scale=.17]{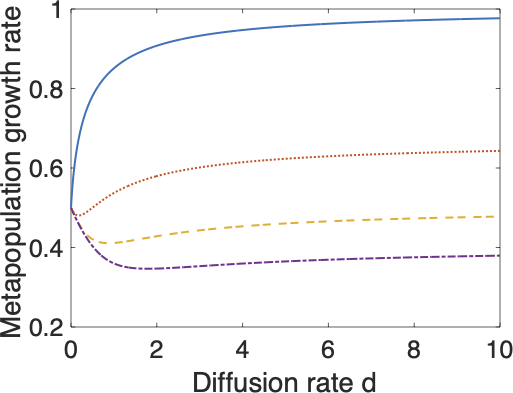}
             \caption{}
         \label{fig:b1}
     \end{subfigure}
       \begin{subfigure}[b]{0.32\textwidth}
         \centering
                    \includegraphics[scale=.17]{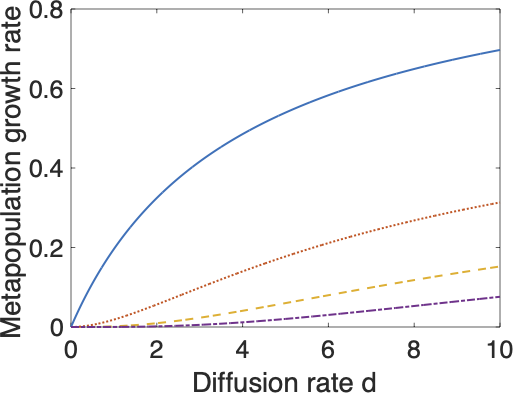}

      \caption{}
         \label{fig:c1}
     \end{subfigure}
     \caption{The metapopulation growth rate as a function of diffusion rate  $d$ for an $n$ patch straight stream when the only patch with positive growth rate $r$ is at the upstream end. The number of patches $n$ vary from $n=2$ (solid curves), $n=3$ (dotted curves), $n=4$ (dashed curves) and $n=5$ (dash-dotted curves). In these figures, we take $r = 2$. The values of $q$ vary from $q = 0.5$ in (a), $q = 1.5$ in (b), to $q = 10$ in (c).}
     \label{fig:chainlengthB}
     \vskip -20pt
\end{figure}

\section{Concluding remarks}\label{section-conclusion}
For all stream networks considered, we found that the largest network biomass is achieved when resources are concentrated in the most upstream patches. Meanwhile, to maximize the metapopulation growth rate, it seems better to concentrate resources in one of the downstream patches.
We also observed that the  drift-diffusion ratio $q/d$ could promote the network biomass if source patches  are located upstream. 
Moreover, the lower bound of the metapopulation growth rate is increasing in $q/d$ if  source patches are located downstream. These observations suggest that a larger drift-diffusion ratio is beneficial for the species in some aspects. This is in contrast with the results proved in \cite{chen2022invasion, chen2021} which state that, for two competing species with equal growth rates and carrying capacities, the species with a smaller drift rate $q$ or a larger diffusion rate $d$ will out-compete the other species in a straight stream network.

We presented two different approaches to study the question of how to maximize the network biomass in a stream network. The first approach, applied to the three-node networks in Section \ref{section-biomass},
makes use of sign patterns imposed by assumptions on the stream flow. Meanwhile, the second approach takes advantage of monotonicity properties of the model. Though this latter approach is more broadly applicable as it does not rely on the explicit structure of the stream network, the sign-pattern approach developed here may prove valuable for studying other types of biological interactions on stream networks. For example, when considering how resource distribution impacts competitive outcomes for a two-species competition model defined on a stream network, the sign pattern approach can be used to determine, based on the distribution of resources for each species, when a semitrivial (exclusion) equilibrium is stable \cite{nguyen2023impact}.

Much of the existing literature related to maximizing the network biomass in spatial models assumes that either the species carrying capacity changes with resource availability or both the carrying capacity and growth rate vary proportionally with changes in resources.  Here we take the alternative approach of fixing the carrying capacity and changing the growth rates.
Though population growth rates have been shown to depend on resource availability in many species (see for example \cite{sibly2002population} and the references therein), this growth is limited by factors such as biological constraints and density dependence. Thus, the growth rate distributions found to maximize the network biomass or metapopulation growth rate may not be biological feasible. From a management perspective, we may want to instead consider the case where, rather than determining the growth rate at each patch, additional resources are provided to supplement the growth in that patch. Such a scenario is described by
\begin{equation}\label{eq-logistic-modified}
u_i' 
= (\bar r_i + r_i)u_i\Big(1-\frac{u_i}{K_i}\Big)+ \sum_{j=1}^n \big(\ell_{ij} u_j - \ell_{ji}u_i\big), \quad i=1,\ldots,n,
\end{equation}
where $\bar r_i$ denotes the intrinsic growth rate in patch $i$ and $r_i$ represents the increase in the growth rate due to resource supplementation. 
Numerical simulations (not included in the current paper) suggest that under some conditions on the parameters $d, q$ and $r$, our findings continue to hold for model \eqref{eq-logistic-modified}. 
In future work, we aim to rigorously study the question of maximizing the network biomass in model \eqref{eq-logistic-modified}.

Another possible modification to our model could be introducing outflows from the most downstream patches (see \cite{chen2021} which considers a stream to a sea or to a lake). Based on numerical simulations, we conjecture that  our results still hold for a large enough drift rate $q$. Additionally, since our model does not allow for patches to be sinks, we could modify the model (see \cite{arino2019number} for more details) to study these scenarios. Finally, we will also consider how the distribution of resources may impact the maximum sustainable yield when harvesting is incorporated into a stream network.

\bigskip

\noindent{\bf Declarations of interest}

None.\medskip

\noindent{\bf Acknowledgments}

This work was partially supported by the National Science Foundation (NSF) through the Grant DMS-1716445 (ZS). 
All authors would like to thank the Department of Mathematics at the University of Central Florida, which receives a funding through the NSF Grant DMS-2132585, for providing support to organize/participate the ``{\it CBMS Conference: Interface of Mathematical Biology and Linear Algebra}'' where this research was initiated and developed. The authors also thank the anonymous reviewers for their helpful suggestions.

\bibliographystyle{plain}
\bibliography{ref}

\newpage 
\appendix
\counterwithout{figure}{section}

\section{Proof of Theorem \ref{theorem growth rate}(ii)}
Let $\rho(d, q, \bm r)$  be the metapopulation growth rate. Denote $S=\{(r_1, r_2, r_3)\in\mathbb{R}^3_+: r_1+r_2+r_3=r\}$. Fix a small $\epsilon>0$. Let

$$
S_3=\{(r_1, r_2, r_3)\in S: r_1\ge q \ \ \text{or}\ \ |r_2-r_3|\le \epsilon\}
$$
and 

$$
S_4=\{(r_1, r_2, r_3)\in S: r_1\le 3q/2 \ \ \text{and}\ \ |r_2-r_3|\ge \epsilon/2\}.
$$
Then $\overline{S\backslash S_3}\subset S_4$ and $S_4=S_{41}\cup S_{42}$, where 

$$
S_{41}=\{(r_1, r_2, r_3)\in S: r_1\le 3q/2 \ \ \text{and}\ \ r_2\ge \epsilon/2+r_3\}
$$
$$
S_{42}=\{(r_1, r_2, r_3)\in S: r_1\le 3q/2 \ \ \text{and}\ \ r_3\ge \epsilon/2+r_2\}.
$$

If $\bm r\in S_3$, then $r_1-2q\le r-2q$, $r_2=r-r_1-r_3\le \max\{r-q, (r+\epsilon)/2\}=:a<r$, and $r_3\le a$. Hence, we have

$$
\rho(0, q, \bm r)=\max\{r_1-2q, r_2, r_3\}\le a.
$$
By the continuity of $\rho$ \cite{zedek1965continuity}, there exists $d_3>0$ such that 
\begin{equation}\label{ess1ii}
\rho(d, q, \bm r)<r-\epsilon
\end{equation}
for all $d\in [0, d_3]$ and $\bm r\in S_3$. Since $\rho(0, q, (0, r, 0))=\rho(0,q, (0, 0, r))=r$, there exists $d_4<d_3$ such that 
\begin{equation}\label{ess2ii}
\rho(d, q, (0, 0, r))>r-\epsilon
\end{equation} 
for all $d\in [0, d_4]$. By \eqref{ess1ii}-\eqref{ess2ii} and $(0, r, 0), (0, 0, r)\in S_4$, we have 

$$
\max\{\rho(d, q, \bm r ): \bm r\in S\}= \max\{\rho(d, q, \bm r): \bm r\in S_4\}
$$
for all $d\in [0, d_4]$, i.e., the maximum of $\rho$ is attained in $S_4$.

By symmetry, it suffices to show that the maximum of $\rho$ in $S_{42}$ is attained at $\bm r=(0, 0, r)$ when $d$ is small. 
For any $\bm r\in  S_{42}$, we have $\rho(0, q, \bm r )=r_3$, where $r_3$ is a simple  eigenvalue of $dD+qQ+\text{diag}\{r_i\}$. Therefore, there exists $d_5<d_4$ such that $\rho(d,q, \bm r)$ is analytic for $d\in [0, d_5]$ and $\bm r\in S_{42}$. Hence, the derivatives of $\rho$ are continuous for $d\in [0, d_5]$ and $\bm r\in  S_{42}$.

Let $h=(-1, 0, 1)$ or $(0, -1, 1)$. It is easy to compute the directional derivatives of $\rho$ with respect to $\bm r$:
$$
D_h\rho(0, q, \bm r)=1, \ \text{for any}\ \bm r\in S_{42}.
$$
By continuity, there exists $d^{**}<d_5$ such that $D_h\rho(d, q, \bm r)>0$ for any $0\le d\le d^{**}$ and $\bm r\in S_{42}$. Hence, the maximum of $\rho(d,q,\bm r)$ in $S_{42}$ is attained at $\bm r=(0, 0, r)$ for any $0\le d\le d^{**}$.

\section{Configurations for stream networks with 4 nodes}
In Figure \ref{fig2} we show all homogeneous flow stream networks with four nodes.
\begin{figure}[p]
\resizebox{.45\totalheight}{!}{
\centering
\begin{tikzpicture}
\begin{scope}[every node/.style={draw}, node distance= 1.5 cm]

    \node[circle] (1) at (-4-1.5,   0-0) {$1$};
    \node[circle] (2) at (-4+0,     0-0) {$2$};
    \node[circle] (3) at (-4+1.5,   0-0) {$3$};
    \node[circle] (4) at (-4+0,     0-4) {$4$};

    \node[circle] (5) at (-4-1.5,   -6-0) {$1$};
    \node[circle] (6) at (-4+1.5,   -6-0) {$2$};
    \node[circle] (7) at (-4+0,     -6-2) {$3$};
    \node[circle] (8) at (-4+0,     -6-4) {$4$};

    \node[circle] (9) at  (0+0,     0-0) {$1$};
    \node[circle] (10) at (0+0,     0-2) {$2$};
    \node[circle] (11) at (0+0,     0-4) {$3$};
    \node[circle] (12) at (0+0,     0-6) {$4$};

    \node[circle] (13) at (4+0,     -6-0) {$1$};
    \node[circle] (14) at (4+0,     -6-2) {$2$};
    \node[circle] (15) at (4-1.5,   -6-4) {$3$};
    \node[circle] (16) at (4+1.5,   -6-4) {$4$};

    \node[circle] (17) at (4+0,     0-0) {$1$};
    \node[circle] (18) at (4-1.5,   0-4) {$2$};
    \node[circle] (19) at (4+0,     0-4) {$3$};
    \node[circle] (20) at (4+1.5,   0-4) {$4$};

    \node[circle] (21) at (4+0,     -12-0) {$1$};
    \node[circle] (22) at (4-1.5,   -12-2) {$2$};
    \node[circle] (23) at (4+1.5,   -12-2) {$3$};
    \node[circle] (24) at (4+0,     -12-4) {$4$};

    \node[circle] (25) at (0,-12) {$1$};
    \node[circle] (26) at (-1.5,-14) {$2$};
    \node[circle] (27) at (1.5,-14) {$3$};
    \node[circle] (28) at (0,-16) {$4$};

    \node[circle] (29) at (-4+0,    -12-0) {$1$};
    \node[circle] (30) at (-4-1.5,  -12-2) {$2$};
    \node[circle] (31) at (-4+1.5,  -12-2) {$3$};
    \node[circle] (32) at (-4+0,    -12-4) {$4$};

    \node[circle] (33) at (-4-1.5,  -18-0) {$1$};
    \node[circle] (34) at (-4+1.5,  -18-0) {$2$};
    \node[circle] (35) at (-4-1.5,  -18-4) {$3$};
    \node[circle] (36) at (-4+1.5,  -18-4) {$4$};

    \node[circle] (37) at (4-1.5,  -18-0) {$1$};
    \node[circle] (38) at (4+1.5,  -18-0) {$2$};
    \node[circle] (39) at (4-1.5,  -18-4) {$3$};
    \node[circle] (40) at (4+1.5,  -18-4) {$4$};
    
\end{scope}
\begin{scope}[every node/.style={fill=white},
              every edge/.style={thick}]
    
    \draw[line width=0.5mm] [->](1) to [bend right=15] node[left=5] {{}} (4); 
    \draw[thick, dashed] [<-](1) to node[right=4] {{}} (4); 
    \draw[line width=0.5mm] [->](2) to [bend left=15] node[right=5] {{}} (4); 
    \draw[thick, dashed] [<-](2) to node[left=4] {{}} (4); 
    \draw[line width=0.5mm] [->](3) to [bend left=15] node[right=5] {{}} (4); 
    \draw[thick, dashed] [<-](3) to node[left=4] {{}} (4); 

    \draw[line width=0.5mm] [->](5) to [bend right] node[left=5] {{}} (7); 
    \draw[thick, dashed] [<-](5) to node[right=4] {{}} (7); 
    \draw[line width=0.5mm] [->](6) to [bend left] node[right=5] {{}} (7); 
    \draw[thick, dashed] [<-](6) to node[left=4] {{}} (7); 
    \draw[line width=0.5mm] [->](7) to [bend left] node[right=5] {{}} (8); 
    \draw[thick, dashed] [<-](7) to node[left=4] {{}} (8); 

    \draw[line width=0.5mm] [->](9) to [bend right] node[left=0.1] {{}} (10); 
    \draw[line width=0.5mm] [->](10) to [bend right] node[left=0.1] {{}} (11); 
    \draw[line width=0.5mm] [->](11) to [bend right] node[left=0.1] {{}} (12); 
    \draw[thick, dashed] [<-](9) to [bend left] node[right=0.1] {{}} (10); 
    \draw[thick, dashed] [<-](10) to [bend left] node[right=0.1] {{}} (11); 
    \draw[thick, dashed] [<-](11) to [bend left] node[right=0.1] {{}} (12); 

    \draw[line width=0.5mm] [->](13) to [bend left] node[right=5] {{}} (14); 
    \draw[thick, dashed] [<-](13) to node[left=4] {{}} (14); 
    \draw[line width=0.5mm] [->](14) to [bend right] node[left=5] {{}} (15); 
    \draw[thick, dashed] [<-](14) to node[right=4] {{}} (15); 
    \draw[line width=0.5mm] [->](14) to [bend left] node[right=5] {{}} (16); 
    \draw[thick, dashed] [<-](14) to node[left=4] {{}} (16); 

    \draw[line width=0.5mm] [->](17) to [bend right=15] node[left=5] {{}} (18); 
    \draw[thick, dashed] [<-](17) to node[right=4] {{}} (18); 
    \draw[line width=0.5mm] [->](17) to [bend left=15] node[right=5] {{}} (19); 
    \draw[thick, dashed] [<-](17) to node[left=4] {{}} (19); 
    \draw[line width=0.5mm] [->](17) to [bend left=15] node[right=5] {{}} (20); 
    \draw[thick, dashed] [<-](17) to node[left=4] {{}} (20); 

    \draw[line width=0.5mm] [->](25) to [bend right] node[left=5]{{}} (26); 
    \draw[thick, dashed] [<-](25) to node[right=4] {{}} (26); 
    \draw[line width=0.5mm] [->](25) to [bend left] node[right=5] {{}} (27); 
    \draw[thick, dashed] [<-](25) to node[left=4] {{}} (27); 
    \draw[line width=0.5mm] [->](26) to [bend right] node[left=5] {{}} (28); 
    \draw[thick, dashed] [<-](26) to node[right=4] {{}} (28); 
    \draw[line width=0.5mm] [->](27) to [bend left] node[right=5] {{}} (28); 
    \draw[thick, dashed] [<-](27) to node[left=4] {{}} (28); 

    \draw[line width=0.5mm] [->](21) to [bend right] node[left=5] {{}} (22); 
    \draw[thick, dashed] [<-](21) to node[right=4] {{}} (22); 
    \draw[line width=0.5mm] [->](22) to [bend right] node[left=5] {{}} (24); 
    \draw[thick, dashed] [<-](22) to node[right=4] {{}} (24); 
    \draw[line width=0.5mm] [->](23) to [bend left] node[right=5] {{}} (24); 
    \draw[thick, dashed] [<-](23) to node[left=4] {{}} (24); 

    \draw[line width=0.5mm] [->](29) to [bend right] node[left=5] {{}} (30); 
    \draw[thick, dashed] [<-](29) to node[right=4] {{}} (30); 
    \draw[line width=0.5mm] [->](29) to [bend left] node[right=5] {{}} (31); 
    \draw[thick, dashed] [<-](29) to node[left=4] {{}} (31); 
    \draw[line width=0.5mm] [->](30) to [bend right] node[left=5] {{}} (32); 
    \draw[thick, dashed] [<-](30) to node[right=4] {{}} (32); 

    \draw[line width=0.5mm] [->](33) to [bend right] node[left=5] {{}} (35); 
    \draw[thick, dashed] [<-](33) to node[right=4] {{}} (35); 
    \draw[line width=0.5mm] [->](34) to [bend right] node[right=5] {{}} (35); 
    \draw[thick, dashed] [<-](34) to node[left=4] {{}} (35); 
    \draw[line width=0.5mm] [->](34) to [bend right] node[left=5] {{}} (36); 
    \draw[thick, dashed] [<-](34) to node[right=4] {{}} (36); 

    \draw[line width=0.5mm] [->](37) to [bend right] node[left=5] {{}} (39); 
    \draw[thick, dashed] [<-](37) to node[right=4] {{}} (39); 
    \draw[line width=0.5mm] [->](38) to [bend right] node[right=5] {{}} (39); 
    \draw[thick, dashed] [<-](38) to node[left=4] {{}} (39); 
    \draw[line width=0.5mm] [->](38) to [bend right] node[left=5] {{}} (40); 
    \draw[thick, dashed] [<-](38) to node[right=4] {{}} (40); 
    \draw[line width=0.5mm] [->](37) to [bend right] node[left=5] {{}} (40); 
    \draw[thick, dashed] [<-](37) to node[right=4] {{}} (40); 

\end{scope}
\end{tikzpicture}
}
\caption{Homogeneous flow stream networks with four nodes. The solid edges have weight $d+q$ and the dashed edges have weight $d$.}\label{fig2}
\end{figure}
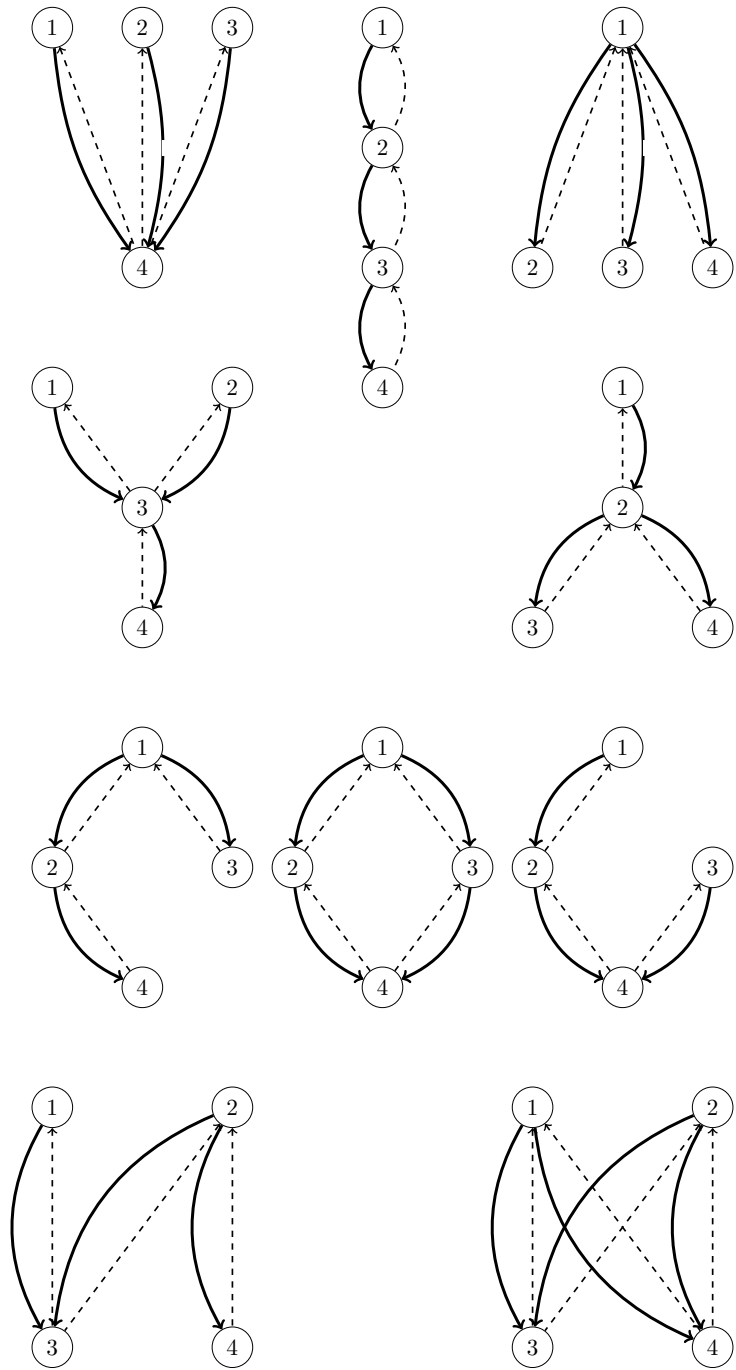

\section{Proof of Theorem \ref{theorem:growth_n}(i)}
Let $\rho(d, q, \bm r)$  be the metapopulation growth rate. Let $V=\{1, ..., n\}$ be the set of all the nodes. Suppose that  there are $k$  downstream end nodes: $n-k+1, n-k+2, \dots, n$.  If $k=1$, then the network is a straight network and the proof is similar to the proof of Theorem \ref{theorem growth rate}(i). Therefore, we  assume $k\ge 2$.
Let $V_e=\{n-k+1, \dots, n\}$ and $V_o=V\backslash V_e=\{1, \dots, n-k\}$.  Denote $S=\{\bm r=(r_1, \dots, r_n)\in\mathbb{R}^n_+: \sum_{i=1}^n r_i=r\}$.  Fix a small $\epsilon>0$. Let $S_5=S_5^1\cup S_5^2$, where

$$
S_5^1=\{\bm r\in S: r_i\ge q/2 \ \text{for some}\  i\in V_o\}
$$
and 

$$
S_5^2=\{\bm r\in S:  r_j-r_l\le \epsilon \ \text{for some} \ l, j \ (l\neq j)\in V_e, 
\ \text{where} \ r_j=\max_{i\in V_e} r_i\}.
$$
Then $\overline{S\backslash S_5} \subset \cup_{j=n-k+1}^n S_{6j}:=S_6$, where 

$$
S_{6j}=\{\bm r\in S: r_i\le 3q/4   \ \text{for all}\  i\in V_o \ \ \text{and}\ \ r_j\ge \epsilon/2+r_l\ \text{for all} \ l \ (l\neq j)\in V_e\}
$$
for $j=n-k+1, \dots, n$.

It is easy to see that

$$
\rho(0, q, \bm r)=\max\{r_1-a_1q, \dots, r_{n-k}-a_{n-k}q, r_{n-k+1}, \dots, r_n\},
$$
where $a_i$ is the number of adjacent downstream nodes of node $i$ for $i\in V_o$. \

If $\bm r\in S_5$, then $\rho(0, q, \bm r)\le \max\{r-q/2, (r+\epsilon)/2\}=:a<r$.
By the continuity of $\rho$ \cite{zedek1965continuity}, there exists $d_5>0$ such that 
\begin{equation}\label{ess1iiC}
\rho(d, q, \bm r)<r-\epsilon
\end{equation}
for all $d\in [0, d_5]$ and $\bm r\in S_5$. Since $\rho(0, q, (0,\dots, 0, r))=r$, there exists $d_6<d_5$ such that 
\begin{equation}\label{ess2iiC}
\rho(d, q,(0, \dots, 0, r))>r-\epsilon
\end{equation} 
for all $d\in [0, d_6]$. By \eqref{ess1iiC}-\eqref{ess2iiC} and $(0, \dots, 0, r)\in S_6$, we have 

$$
\max\{\rho(d, q, \bm r): \bm r\in S\}= \max\{\rho(d, q, \bm r): \bm r\in S_6\}
$$
for all $d\in [0, d_6]$, i.e., the maximum of $\rho$ is attained in $S_6$.

It suffices to show that the maximum of $\rho$ in $S_{6j}$, $j=n-k+1, \dots, n$, is attained when all the resources are concentrated at node $j$  if $d$ is small. For brevity, we only consider the case $j=n$ in the following. 
For any $\bm r\in  S_{6n}$, we have $\rho(0, q,\bm r)=r_n$, where $r_n$ is a simple  eigenvalue of $dD+qQ+\text{diag}\{r_i\}$. Therefore, there exists $d_7<d_6$ such that $\rho(d,q, \bm r)$ is analytic for $d\in [0, d_7]$ and $\bm r\in S_{6n}$. Hence, the derivatives of $\rho$ are continuous for $d\in [0, d_7]$ and $\bm r\in  S_{6n}$.

Let $h=(-1, 0, \dots, 0, 1)$, $(0, -1, 0, \dots, 0, 1)$, \dots, or $(0, \dots, 0, -1, 1)$. It is easy to compute the directional derivatives of $\rho$ with respect to $\bm r$:

$$
D_h\rho(0,q, \bm r)=1, \ \text{for any}\ \bm r\in S_{6n}.
$$
By continuity, there exists $\tilde{d}^{*}<d_7$ such that $D_h\rho(d, q,\bm r)>0$ for any $0\le d\le \tilde{d}^{*}$ and $\bm r\in S_{6n}$. Hence, the maximum of $\rho(d,q, \bm r)$ in $S_{6n}$ is attained at $\bm r=(0,\dots, 0, r)$ for any $0\le d\le \tilde{d}^{*}$.

\end{document}